\documentclass[12pt]{article}
\usepackage{verbatim}
\usepackage{hyperref}
\usepackage{amsmath}
\usepackage{enumerate, amsthm}
\usepackage{mathrsfs}
\usepackage{xcolor}
\usepackage[margin=1 in]{geometry}
\usepackage{mathtools}
\usepackage{tikz}
\usetikzlibrary{patterns}
\usetikzlibrary{arrows}
\usetikzlibrary{arrows.meta}
\usepackage{ifthen}
\usepackage{bm}
\usepackage{amssymb}
\usepackage{centernot}
\usepackage{calrsfs}
\usepackage{caption}
\usepackage{subcaption}
\numberwithin{equation}{section}
\usepackage{extpfeil} 
\usepackage{cancel} 
\usepackage[shortlabels]{enumitem}
\usepackage{bbm, dsfont} 

\theoremstyle{plain}
\newtheorem{theorem}{Theorem}[section]
\newtheorem{lemma}[theorem]{Lemma}
\newtheorem{corollary}[theorem]{Corollary}
\newtheorem{proposition}[theorem]{Proposition}
\newtheorem{question}[]{Question}
\theoremstyle{definition}
\newtheorem{definition}[theorem]{Definition}

\theoremstyle{remark}
\newtheorem{remark}[theorem]{Remark}

\mathtoolsset{showonlyrefs}

\renewcommand{\P}{\mathbb P}
\newcommand{\E}{\mathbb E}
\renewcommand{\H}{\mathbb H}

\newcommand{\R}{\mathbb R}

\newcommand{\Z}{\mathbb Z}
\newcommand{\N}{\mathbb N}
\newcommand{\ep}{\varepsilon}

\renewcommand{\d}{\mathrm{d}}
\newcommand{\de}{\mathrm{d}_{\mathbb{R}^d}}
\newcommand{\dm}{\mathrm{d}_{M}}

\newcommand{\wt}{\widetilde}
\newcommand{\one}{\mathbbm{1}}

\DeclareMathOperator{\PPP}{PPP}
\DeclareMathOperator{\diam}{diam}
\DeclareMathOperator{\im}{im}
\DeclareMathOperator{\dom}{dom}

\mathtoolsset{showonlyrefs}

\newcommand{\lr}[4]{#3\xleftrightarrow[#1]{#2} #4}
\newcommand{\nlr}[4]{#3\centernot{\xleftrightarrow[#1]{#2}} #4}

\usepackage{mathtools}

\DeclareMathAlphabet{\pazocal}{OMS}{zplm}{m}{n}

\title{High-intensity Voronoi percolation on manifolds}
\author{Tillmann B\"uhler\footnotemark[1]\footnote{Karlsruhe Institute of Technology, tillmann.buehler@kit.edu}, Barbara Dembin\footnotemark[2]\footnote{CNRS and University of Strasbourg, barbara.dembin@math.unistra.fr}, Ritvik Ramanan Radhakrishnan\footnotemark[3]\footnote{ETH Z\"{u}rich, ritvik.radhakrishnan@math.ethz.ch}, \and Franco Severo\footnotemark[4]\footnote{Institut Camille Jordan, severo@math.univ-lyon1.fr}}

\date{}

\begin{document}
\thispagestyle{empty}

\maketitle

\begin{abstract} 
We study Voronoi percolation on a large class of $d$-dimensional Riemannian manifolds, which includes the hyperbolic spaces $\H^d$, $d\geq 2$. We prove that as the intensity $\lambda$ of the underlying Poisson point process tends to infinity, both critical parameters $p_c(M,\lambda)$ and $p_u(M,\lambda)$ converge to the Euclidean critical parameter $p_c(\mathbb{R}^d)$. 
This extends a recent result of Hansen \& Müller in the special case $M=\mathbb{H}^2$ to a general class of manifolds of arbitrary dimension. A crucial step in our proof, which may be of independent interest, is to show that if $M$ is simply connected and one-ended, then embedded graphs induced by a general class of tessellations on $M$ have connected minimal cutsets. In particular, this result applies to $\varepsilon$-nets, allowing us to implement a ``fine-graining'' argument.
\end{abstract}

\section{Introduction}\label{sec:intro}

Let $(M,g)$ be a connected, complete, non-compact, $d$-dimensional Riemannian manifold and let $\dm$ and $\mu_{M}$ be the associated distance function and volume measure, respectively. 
Given $\lambda>0$, let $\eta_{\lambda}$ be a Poisson point process with intensity $\lambda\mu_{M}$ on $M$, and consider the associated Voronoi tessellation whose cells are defined as
\begin{equation*}
C(x;\eta_\lambda)\coloneqq\{y\in M:\, \dm(y,x)\leq \dm(y,\eta_\lambda)\}, ~~ x\in{\eta_\lambda}.
\end{equation*}
Note that this indeed defines a partition of $M$ since $\eta_\lambda$ is a.s.\ a locally finite set (i.e., $\eta_\lambda\cap B$ is finite for any bounded $B\subset M$).
Under mild additional conditions on $M$, these cells are all compact almost surely (see Appendix~\ref{app:bounded cells} for a proof).
In Voronoi percolation, given $p\in[0,1]$, we independently color each cell white or black with probabilities $p$ and $1-p$, respectively. We are then interested in understanding the connected components (or clusters) of the random set corresponding to the union of the white cells, i.e.,
$$\pazocal{V}_{M}(\lambda,p) \coloneqq \bigcup_{x\in \eta_\lambda \text{ white}} C(x;\eta_\lambda).$$
Let $\pazocal{N}_{M}(\lambda,p)$ be the number of infinite (i.e., unbounded) clusters of $\pazocal{V}_{M}(\lambda,p)$. The critical parameter for the emergence of an infinite cluster is defined as
\begin{equation*}
    p_c(M,\lambda)\coloneqq\sup\{p\in[0,1]: \P(\pazocal{N}_{M}(\lambda,p)\geq 1)=0\}.
\end{equation*}
Since $\{\pazocal{N}_{M}(\lambda,p)\geq 1\}$ is an increasing tail event, there is no infinite cluster almost surely for $p<p_c$, while there is at least one infinite cluster almost surely for $p>p_c$. Note that Voronoi percolation can also be viewed as standard Bernoulli site percolation on the so-called Voronoi (or Delaunay) graph $\pazocal{G}_{M}(\eta_\lambda)=(V(\eta_\lambda),E(\eta_\lambda))$ induced by the Voronoi tessellation associated to $\eta_\lambda$, where $V(\eta_\lambda)=\eta_\lambda$ and $\{x,y\}\in E(\eta_\lambda)$ if and only if $C(x;\eta_\lambda)\cap C(y,\eta_\lambda)\neq\emptyset$.

The special case $M=\R^d$ (with the standard Euclidean metric) has been studied extensively. This model was first investigated in \cite{VW90} in the context of first-passage percolation, and the existence of a non-trivial phase transition (i.e.,~$p_c\in(0,1)$) was obtained in \cite{BBQ05}.
It is easy to see that the following scaling property holds: $\pazocal{V}_{\R^d}(\lambda,p)$ has the same distribution as $\lambda^{-1/d} \pazocal{V}_{\R^d}(1,p)$. In particular, $p_c(\R^d,\lambda)$ does not depend on $\lambda$; we simply write $p_c(\R^d)$ for this common value. Furthermore, since Euclidean space is amenable, one can show that $\pazocal{N}_{\R^d}(p)\leq1$ almost surely for every $p\in[0,1]$ (see, e.g., \cite[p.~278]{BR_percolation}). In the planar case $d=2$, one can describe the model more precisely using self-duality. Bollob\'as \& Riordan \cite{BR_Voronoi_percolation} showed that
\begin{equation}\label{eq:p_c=1/2}
	p_c(\R^2)=1/2.
\end{equation}
Furthermore, there is no infinite component at the critical point \cite{zvavitch}. An even more detailed analysis of the critical behavior on $\R^2$ has been the subject of several subsequent works, e.g.,~\cite{Tas16,AGMT12,AB18,Van19}.  In higher dimensions, the behavior at criticality remains an extremely challenging open problem. In contrast, the subcritical ($p<p_c$) and supercritical ($p>p_c$) regimes are now relatively well understood due to \emph{sharpness results} of Duminil-Copin, Raoufi \& Tassion \cite{DRT17} and Dembin \& Severo \cite{DS25}, respectively. These works are crucial ingredients in the proof of our main result, Theorem~\ref{thm:main} below. We explain this in more detail later.

For general manifolds, when an infinite cluster exists, it may or may not be unique. We consider the uniqueness critical parameter
\begin{equation*}
	p_u(M,\lambda)\coloneqq\sup\{p\in[0,1]:\P(\pazocal{N}_{M}(\lambda,p)=1)<1\}.
\end{equation*}
By definition, there exists a unique infinite cluster almost surely for every $p>p_u~(\geq p_c)$. What happens for $p\in[p_c,p_u]$ is less clear without any further assumption on $M$. 
If one assumes that $M$ is homogeneous (i.e.,~its isometry group $\mathrm{Isom}(M)$ acts transitively on $M$), then ergodicity implies that the $\mathrm{Isom}(M)$-invariant random variable $\pazocal{N}_{M}(\lambda,p)$ takes a deterministic value almost surely, which in turn, by a standard insertion tolerant argument, is either $0$, $1$ or $\infty$ (see the proofs of Lemma 7.4 and Theorem 7.5 of \cite{LP16}).\footnote{In fact, it is enough to assume that the orbits of $\mathrm{Isom}(M)$ are unbounded.} If $M$ admits a transitive and unimodular group of isometries, then $\pazocal{G}_{M}(\eta_\lambda)$ turns out to be a unimodular random network, as defined in \cite{aldouslyons} (see Example 9.5 therein). Thus, several results from \cite{aldouslyons} apply in this context. 
For instance, it follows directly from \cite[Theorem 6.7]{aldouslyons} that $\pazocal{N}_{M}(\lambda,p)=\infty$ almost surely for every $p\in(p_c(M,\lambda),p_u(M,\lambda))$. For certain symmetric spaces, the uniqueness phase is characterized by long-range order, see~\cite[Theorem 1.9]{GR25}.

As mentioned above, we always have uniqueness on $\R^d$, so $p_u(\R^d)=p_c(\R^d)$. For other manifolds however, we may have $p_c(M,\lambda)<p_u(M,\lambda)$.
An emblematic example for which this has been established is the hyperbolic plane $\H^2$. \mbox{Benjamini} \& Schramm \cite{BS01} proved, among other things, that for every $\lambda>0$ we have 
\begin{equation}\label{eq:H^2_pu=1-pc}
	p_u(\H^2,\lambda)= 1-p_c(\H^2,\lambda)>1/2,
\end{equation}
which implies $p_c(\H^2,\lambda)<p_u(\H^2,\lambda)$.
Proving the existence of a non-uniqueness phase in greater generality remains a very interesting and challenging open problem, even for Bernoulli percolation on transitive graphs. We discuss this in more detail later in the introduction.

Intuitively, for large intensity $\lambda$, the Poisson--Voronoi tessellation on $M$ should look \emph{locally} (but never globally for a fixed $\lambda$) like a Poisson--Voronoi tessellation on $\R^d$. In fact, we can interpret $\lambda^{1/d}$ as the inverse curvature of the associated graph. It is widely believed, but rather hard to prove, that percolation critical points are continuous functions of the underlying graph in an appropriate sense. An important example of such a statement is the famous \emph{locality conjecture} on transitive graphs, which was proposed by Schramm and recently proved by Easo \& Hutchcroft \cite{EH23} (see also \cite{MT17locality,CMT23locality} for earlier progress on this problem, and \cite{BPT17} for counterexamples in the context of unimodular random graphs).
It is then natural to expect that $p_c(M,\lambda)\to p_c(\R^d)$ as $\lambda\to \infty$ for any reasonable $d$-dimensional manifold $M$. For instance, Benjamini \& Schramm \cite{BS01} conjectured that 
$$p_c(\H^2,\lambda)\to 1/2~(=p_c(\R^2))~ \text{ as } \lambda\to \infty,$$
which was recently proved by Hansen \& Müller \cite{HM_large_int}. Notice that by \eqref{eq:H^2_pu=1-pc} it directly follows that $p_u(\H^2,\lambda)\to 1/2$ as well.

Our main theorem shows that this locality property of $p_c(M,\lambda)$ and $p_u(M,\lambda)$ holds for a very general class of manifolds of arbitrary dimension. In particular, this class includes all hyperbolic spaces $\H^d$, $d\geq2$, thus proving \cite[Conjecture 4.6]{HM24}.

\begin{theorem}\label{thm:main}
Let $M$ be a connected, complete, $d$-dimensional Riemannian manifold satisfying the following conditions 
\begin{enumerate}[i)]

    \item\label{thm cond: global injectivity radius} $M$ has positive global injectivity radius \(r_\text{inj}(M)>0\),\footnote{See Section \ref{sec:localcomparison} or \cite[p.~165-6]{LeeRM} for a precise definition.}

    \item\label{thm cond: uniform bound on curvature} the sectional curvatures of M are uniformly bounded from above and from below,\footnote{see, e.g., \cite[p.~250]{LeeRM} for a definition and an introduction to sectional curvatures.}
    
    \item\label{thm cond : connected+one-ended} $M$ is simply connected and one-ended,\footnote{We call $M$ \emph{one-ended} when $M\setminus B$ has exactly one unbounded connected component for any bounded $B\subset M$.}
    
    \item\label{thm cond: log expansion} there exists a log-expanding, bi-infinite path in $M$, i.e., a Lipschitz continuous path $\gamma:(-\infty,\infty)\to M$ satisfying, for some constants $c,C\in(0,\infty)$, the inequality
        \begin{equation}
            \dm(\gamma(t),\gamma(s))\geq c\log |t-s|-C ~~~ \text{ for all } t,s\in(-\infty,\infty).
        \end{equation} 
\end{enumerate} 

Then, both $p_c(M,\lambda)$ and $p_u(M,\lambda)$ converge to $p_c(\R^d)$ as $\lambda\rightarrow \infty$.
\end{theorem}

Hypotheses \ref{thm cond: global injectivity radius} and \ref{thm cond: uniform bound on curvature} are fairly standard in Riemannian geometry. They allow us to control the volume growth of balls and obtain uniform local comparison across different locations. The two other hypotheses cannot be removed; a discussion about the assumptions is presented later in this introduction.

As explained in more detail below, a crucial step in the proof of Theorem~\ref{thm:main} is to show that graphs induced by $\ep$-nets of $M$ have connected minimal cutsets.
The following result, which may be of independent interest, implies that this is the case for a very general class of embedded graphs -- see \cite{BB99,Tim07} for results of a similar flavor.

We quickly recall the definition of cutsets. Given a graph $G=(V,E)$, $x\in V$ and $y\in V \cup\{\infty\}$, we say that $\Gamma\subset V$ is a \emph{cutset} separating $x$ from $y$ if any path from $x$ to $y$ in $G$ intersects $\Gamma$. We say that $\Gamma $ is a \emph{minimal cutset} if for all $w\in\Gamma$ the set $\Gamma\setminus \{w\}$ is no longer a cutset.

\begin{theorem}\label{thm: connected cutsets}
Let $M$ be simply connected. Let $G=(V,E)$ be a graph such that $V\subset M$. Assume that to each $x\in V$ we can associate an open connected set $O(x)\subset M$ containing $x$ such that
\begin{enumerate}[i)]
\item $\bigcup_{x\in V}O(x)=M$,
\item $O(x)\cap O(y)\ne \emptyset $ if and only if $\{x,y\}\in E$ (for $x\neq y$).
\end{enumerate}
    Then for all $x,y\in V$, any minimal cutset separating $x$ from $y$ is connected in $G$.
    
    If we further assume that $G$ is one-ended and locally finite\footnote{We call a graph $G=(V,E)$ \emph{one-ended} if for any finite set $\Gamma\subset V$ the induced subgraph on $V\setminus \Gamma$ has a unique infinite connected component.
    We call \(G\) \emph{locally finite} when each of its vertices has finite degree.}
    (which holds, for instance, if $M$ itself is one-ended, the sets $O(x)$ are all bounded and any bounded set \(B \subset M\) intersects at most finitely many of the sets \(O(x)\)), then any minimal cutset separating $x$ from infinity is connected in $G$.
\end{theorem}

By defining appropriate open neighborhoods of the (closed) cells $C(x;\eta_\lambda)$, we obtain that the conclusion of Theorem~\ref{thm: connected cutsets} holds for the Voronoi graph $\pazocal{G}_{M}(\eta_\lambda)$ almost surely, see Corollary~\ref{cor:connected cutsets closed}. 
 It turns out that for graphs with connected cutsets, one can upper bound $p_u$ (in fact, an even larger critical parameter $\overline{p}_u$) in terms of $p_c$, see Theorem~\ref{thm:bound on p_u}.
In particular, this applies to $\pazocal{G}_{M}(\eta_\lambda)$ itself, and we obtain the following.

\begin{theorem}\label{thm:bound on p_u for Voronoi}
Assume that $M$ is simply connected and one-ended, has a positive global injectivity radius, and all sectional curvatures of $M$ are bounded from above. Then, for every $\lambda>0$, either $p_c(M,\lambda)=p_u(M,\lambda)> 1/2$ or
\begin{equation}\label{eq:p_u leq 1-p_c}
	p_u(M,\lambda)\leq 1-p_c(M,\lambda).
\end{equation}
\end{theorem}

\paragraph{About the proofs.}
We now discuss the proof of Theorem~\ref{thm:main}. The first step is to establish the aformentioned local resemblance between the Poisson--Voronoi tesselation on $M$ and $\R^d$ for high intensities. More precisely, we show in Proposition~\ref{prop:comparison} that for any $x\in M$, we can couple the Poisson--Voronoi tessellations of intensity $\lambda$ on a ball of radius $\ep=\Theta(\lambda^{-1/d})$ around $x$ in $M$ and on a corresponding ball in $\R^d$ in such a way that the associated graphs agree with high probability as $\lambda\to\infty$. While it is easy to couple the two point processes in such a way that they coincide with high probability, some extra care is required to guarantee that the adjacency structure of the cells is likely to remain stable under the change of metric. 

The second step is to extend this local comparison to the global scale and establish properties such as the existence (and uniqueness) of an unbounded component. For this, we need a \emph{finite-size criterion} for the \emph{entire} subcritical and supercritical phases of Voronoi percolation in $\mathbb{R}^d$, which will allow us to implement a \emph{``fine-graining''} argument (unlike in the classical coarse-graining arguments in percolation theory, here we deal with small balls, which explains the term ``fine''). In the subcritical regime, this finite-size criterion is that large annuli are crossed with low probability. In the supercritical regime, it is that large annuli are crossed by a \emph{unique} component with high probability.  
The fact that these finite-size criteria hold up to the critical point is often referred to as a \emph{sharpness result}. As mentioned above, these results were recently established for Voronoi percolation in $\R^d$ in \cite{DRT17} (for the subcritical phase) and \cite{DS25}  (for the supercritical phase). 

We combine the results described in the previous two paragraphs to deduce that for every $p<p_c(\R^d)$ (resp.,~$p>p_c(\R^d)$), a certain event $F_\ep^{sub}$ (resp.,~$F_\ep^{sup}$), defined on $B_{M}(x,\ep)$ and corresponding to a subcritical (resp.,~supercritical) finite-size criterion, holds with high probability for $\pazocal{V}_{M}(\lambda,p)$, provided that $\lambda$ is large enough. This implies the absence (resp.,~existence and uniqueness) of an unbounded component for $\pazocal{V}_{M}(\lambda,p)$, which we prove using a fine-graining argument: we construct a graph $G_\ep$ embedded in $M$ with mesh-size $\ep$ and define a renormalized percolation model on $G_\ep$ where a site $x$ is open depending on whether the event $F_\ep^{sub}$ (resp. $F_\ep^{sup}$) occurs around $x$. We then compare this model (in the sense of stochastic domination) with a Bernoulli percolation having parameter close to $0$ (resp.,~$1$) on $G_\ep$. 

Our problem is then reduced to constructing an appropriate family of embedded graphs $(G_\ep)_{\ep>0}$, of mesh-size $\ep$, in such a way that their Bernoulli percolation critical points $p_c(G_\ep)$ and $p_u(G_\ep)$ are uniformly bounded away from $0$ and $1$. (In fact, because uniqueness is a non-monotonic event and we rely on stochastic domination, we need to upper bound a stronger uniqueness critical parameter $\overline{p}_u(G_\ep)$, above which the (unique) infinite cluster has only finite holes, i.e.,~the complement of the infinite cluster consists of only finite components.) For Euclidean space $\R^d$, this is easy as one can take $G_\ep=\ep\Z^d$, which are all isomorphic to each other (in particular, they have uniformly bounded critical points). In general, manifolds may not admit such nice tessellations with mesh-size tending to zero -- this is the case even for $\H^d$ ({see, e.g.,~\cite[Section 2]{martin2013Kleinian}).
Due to Theorem~\ref{thm: connected cutsets}, we can simply take $G_\ep$ to be the graph induced by arbitrary $\ep$-nets of $M$. Indeed, we can combine connectivity of cutsets with other basic geometric properties of these graphs, such as uniformly bounded degrees and existence of a log-expanding path, to prove via a Peierls argument that $p_c(G_\ep)$ and $\overline{p}_u(G_\ep)$ are indeed uniformly bounded away from $0$ and $1$ for $\ep\in(0,1]$ -- see Theorem~\ref{thm:discretization_percolation}. In fact, it is enough to bound $p_c(G_\ep)$ since the inequality $\overline{p}_u(G)\leq \max\{p_c(G), 1-p_c(G)\}$ holds for any one-ended graph $G$ with connected minimal cutsets. This is proved in Theorem~\ref{thm:bound on p_u}, from which Theorem~\ref{thm:bound on p_u for Voronoi} follows.
All these results are in Section~\ref{sec:discretization}.

Our work contributes to the broader development of renormalization theory for percolation on general geometries. In particular, while the proof of Theorem~\ref{thm:main} employs a fine-graining approach to analyze the high-intensity regime, in Appendix~\ref{app:phase-transition} we apply a complementary coarse-graining technique to establish the non-triviality of the phase transition for \emph{all} $\lambda > 0$, under the assumptions of Theorem~\ref{thm:main} (and superlinear growth). Furthermore, we believe these techniques could offer useful tools for studying other models on general manifolds, such as Boolean percolation.

\paragraph{Comparison with $\H^2$.}

We highlight here a few key differences between our proof of Theorem~\ref{thm:main} and the proof of Hansen \& M\"uller \cite{HM_large_int} in the special case $M=\H^2$. 
First, due to the relation \eqref{eq:H^2_pu=1-pc}, the convergence \(p_c(\H^2,\lambda) \to 1/2\) immediately implies \(p_u(\H^2,\lambda) \to 1/2\). 
For general manifolds, we require some extra work in the fine-graining argument to deal with \(p_u(M,\lambda)\), that is, we need to (uniformly) upper bound the stronger uniqueness parameter $\overline{p}_u$ described above.
More importantly, in dimension $d=2$ we can reformulate the finite-size criterion for the supercritical regime as the existence of a circuit around an annulus, which is simpler than local uniqueness. In particular, this criterion is monotone, which is not true for local uniqueness. 
Because of non-monotonicity, in order to import local uniqueness from Euclidean to hyperbolic space, we work at scales where the Euclidean and hyperbolic Voronoi graphs are \emph{exactly} the same with high probability. This appropriate scale shrinks as $\lambda\to \infty$, forcing us to consider finer and finer graphs. If we were only dealing with monotonic events, we could directly compare the two processes using domination at a fixed scale, like in \cite{HM_large_int}. As a consequence, the renormalization argument in \cite{HM_large_int} happens on a fixed tiling $G$ of scale $1$, regardless of $\lambda$, which is taken to be the regular $(7,7,7)$-triangulation of $\H^2$. The renormalization argument is successful since $p_c(G)\in(0,1)$, which is easy to prove since both $G$ and its dual $G^*$ have bounded (actually, constant) degree.

\paragraph{Comments and questions.}
We now make a few comments concerning the assumptions for our results, and ask some questions that arise naturally.

Concerning Theorem~\ref{thm:main}, essentially none of the assumptions can be removed. For instance, the three-dimensional cylinder $\R\times \mathbb{S}^2$ satisfies all the assumptions except one-endedness, but clearly $p_c(\R\times \mathbb{S}^2,\lambda)=1$ for all $\lambda>0$.
Simply connectedness cannot be removed either:
consider the plane graph $G$ obtained by taking the $x$-axis and adding arcs between $(k,0)$ and $(-k,0)$ for all $k\in\N$, as in Figure \ref{fig:one-end}.
The (open) parallel set $\{z \colon \d(z,G)<1/3\}$ is one-ended and contains a log-expanding path, but it is not simply connected.
Clearly, it satisfies $p_c(M,\lambda)=1$ for all $\lambda>0$.\footnote{We cheated a bit here by presenting a manifold that is not complete -- the reader is invited to use their creativity to turn $G$ into a complete manifold if they are not convinced.}
We also believe that a manifold with diverging negative curvature can be constructed in such a way that $p_c(M,\lambda)=0$ for all $\lambda>0$.
 A reasonable candidate is $\R^2$ equipped with a Riemannian metric so that the Gaussian curvature at $x$ equals $-(1+|x|^2)$, which exists by \cite[Theorem 4.1]{KW74}.

Even more interestingly, the condition on the existence of a log-expanding path is not only necessary, but the log function is sharp, as illustrated by the following example. Given a non-decreasing function $f:[0,\infty) \to [0,\infty)$ such that $f(0)=0$, consider the surface of revolution given by $M\coloneqq\{(x,y,z)\in \R^3:\, x^2+y^2=f(z)^2, z \geq 0\}$ -- see Figure \ref{fig:revolution}. 
By taking $f$ sufficiently nice and regular,
we can ensure that conditions \ref{thm cond: global injectivity radius}--\ref{thm cond : connected+one-ended} of Theorem~\ref{thm:main} hold. 
Therefore, our theorem applies to such a manifold if we additionally assume that $f(z)\geq c\log z-C$. However, if $f(z)=o(\log z)$, then it is not hard to prove that $p_c(M,\lambda)=1$ for all $\lambda>0$.\footnote{Indeed, for every $\lambda > 0$ and $p<1$, there exists a constant $c=c(\lambda,p)>0$ and a sequence of events $E_1,E_2,\ldots$, so that $E_k$ depends only on the points in $M \cap (\R^2 \times (3k,3k+3))$, the region $M \cap (\R^2 \times [3k+1,3k+2])$ is completely black under $E_k$ and $\P_{\lambda,p}(E_k) \geq e^{-c(\lambda,p) f(3k)}$.
Since $e^{-c f(3k)}$ is not summable in $k\in\N$, the (reverse) Borel-Cantelli Lemma implies that almost surely infinitely many sections $M\cap (\R^2\times[3k+1,3k+2])$ are fully black, so that there is no unbounded white connected component, and thus $p_c(M,\lambda)=1$.}

\begin{figure}
\centering
\begin{minipage}{.4\textwidth}
  \centering
  \vspace{0.4cm}
  \begin{tikzpicture}
    \draw [thick] (-2.8,0) -- (2.8,0);
    \draw [thick,domain=0:180] plot ({0.5*cos(\x)}, {0.5*sin(\x)});
    \draw [thick,domain=0:180] plot ({cos(\x)}, {sin(\x)});
    \draw [thick,domain=0:180] plot ({1.5*cos(\x)}, {1.5*sin(\x)});
    \draw [thick,domain=0:180] plot ({2*cos(\x)}, {2*sin(\x)});
    \draw [thick,domain=0:180] plot ({2.5*cos(\x)}, {2.5*sin(\x)});
    \node at (3.2,0) {$\ldots$};
    \node at (-3.2,0) {$\ldots$};
    \end{tikzpicture}
  \captionof{figure}{The plane graph $G$. By properly thickening $G$, one obtains a smooth manifold satisfying all assumptions of Theorem~\ref{thm:main} except simply connectedness. This manifold satisfies $p_c(M,\lambda)=1$ for all $\lambda>0$.}
  \label{fig:one-end}
\end{minipage}
\hspace{2cm}
\begin{minipage}{.4\textwidth}
  \centering
  \begin{tikzpicture}[scale = .68]
    \draw[gray, dotted, thick] (2.175cm, 1.5) arc (0:180:2.175cm and 0.3cm);

    \draw[gray, thick] (-2.175cm, 1.5) arc (180:360:2.175cm and 0.3cm);
    \draw[samples=50,line width=1pt] plot[domain=-2.35:2.35] (\x, {cosh(\x) - 3});
    \draw[Stealth-Stealth] (0, 1.5) -- (2.175, 1.5);
    \node[scale=1.2] at (1.1, 1.9) {$f(z)$};

    \draw[Stealth-Stealth] (0, -2) -- (0, 1.5);
    \node[scale=1.2] at (0.25, -.1) {$z$};
\end{tikzpicture}
  \captionof{figure}{The surface of revolution $M$. If $f$ is sufficiently nice and regular but $f(z)=o(\log z)$, then $M$ satisfies all assumptions of Theorem~\ref{thm:main} except log-expansion. This manifold satisfies $p_c(M,\lambda)=1$ for all $\lambda>0$.}
  \label{fig:revolution}
\end{minipage}
\end{figure}

All the examples mentioned above involve manifolds for which $p_c(M,\lambda)$ is trivial (i.e., equal to $0$ or $1$) for all $\lambda$. The two following questions arise naturally.

\begin{question}\label{q:pc<1 implies conv}
    Does $p_c(M,\lambda)\in(0,1)$ for some $\lambda>0$ imply $p_c(M,\lambda)\to p_c(\R^d)$ as $\lambda\to\infty$?
    Does the same result hold for $p_u(M,\lambda)$?
\end{question}

\begin{question}\label{q:pc<1}
    Under what conditions can we guarantee that $p_c(M,\lambda)\in (0,1)$? For instance, does isoperimetric dimension $>1$ imply $p_c(M,\lambda)<1$? If $M$ is homogeneous, is it enough to assume that its volume growth is superlinear?
\end{question}

The presumably sufficient conditions we propose in Question~\ref{q:pc<1} above are inspired by \cite[Conjecture 2 and Question 2]{BS96}, in the context of graphs. This conjecture has been established in \cite{DGRSY20}, along with a partial answer to the question -- see also \cite{EST24} for a new proof of a stronger result. In Remark~\ref{rem:exp_cutsets_transience}, we discuss a plausible approach to Question~\ref{q:pc<1} based on \cite{EST24}. We also believe that the property of having a non-trivial phase transition, i.e., $0<p_c(M,\lambda)\leq p_u(M,\lambda)<1$, depends only on $M$ and not on $\lambda$. In Appendix~\ref{app:phase-transition}, we prove that the phase transition is non-trivial for all $\lambda\in(0,\infty)$ under the assumptions of Theorem~\ref{thm:main} and superlinear growth.

For many homogeneous manifolds of dimensions $d\geq3$ (for instance, those of the form $\H^{d_1}\times\R^{d_2}\times K$, with $K$ homogeneous and compact), one can find an isometric copy of $\R^2$ or $\H^2$ in them. One can then consider the induced random tessellations on this $2$-dimensional manifold, which despite not being a Voronoi tessellation itself, should still be quite well behaved. If there is a copy of \(\H^2\) in \(M\), then one could try to apply \cite[Theorem 6.3]{BS01} to the induced tessellation on $\H^2$ to get that $p_c(M,\lambda) < 1/2$ for all $\lambda>0$.
If there is a copy of \(\R^2\) in \(M\), then one may hope that adapting the proof of \cite{BR_Voronoi_percolation} (or alternatively \cite{DRT17,DS25}) to the induced tessellation on $\R^2$ will yield $p_c(M,\lambda)\leq 1/2$ for all $\lambda>0$. If the inclusion $\R^2\subset M$ is proper, one could even hope to prove, by an essential enhancement argument (as in~\cite{aizenman1991strict} or \cite{MS19}), that $p_c(M,\lambda)< 1/2$.
Motivated by these remarks, we ask the following.
\begin{question}\label{q:pc_gap_1/2}
    If $M$ is homogeneous, does $p_c(M,\lambda)<1$ imply $p_c(M,\lambda)\leq 1/2$ for all $\lambda>0$? 
    If one further assumes that $M$ is not $\R^2$, does it hold that $p_c(M,\lambda)< 1/2$?
\end{question}
We remark that such a ``gap at $1$'' has been established for the percolation threshold of Cayley graphs in \cite{PS21b}, but the precise size of the gap (or equivalently the supremum of $p_c(G)<1$) is unknown in this context. A positive answer to Question~\ref{q:pc_gap_1/2} would not only establish such a gap for homogeneous manifolds, but also show that it is equal to $1/2$.

As mentioned above, it is reasonable to expect that $p_c(\R^d)<1/2$ for $d\geq3$ -- by the asymptotics obtained in \cite{BBQ05}, this is known at least for large enough dimensions. 
Theorem~\ref{thm:main} would then imply, in particular, that the inequality \eqref{eq:p_u leq 1-p_c} is strict for large enough $\lambda$ if $M$ has dimension $d\geq3$. For symmetric spaces of connected semisimple Lie groups of higher rank with property (T), it is shown in \cite[Theorem 1.1]{GR25} that $p_u(M,\lambda)\to0$ as $\lambda\to0$, so \eqref{eq:p_u leq 1-p_c} is also strict in this regime.
On the other hand, \eqref{eq:p_u leq 1-p_c} is an equality for both two-dimensional manifolds $\R^2$ and $\H^2$ by \eqref{eq:p_c=1/2} and \eqref{eq:H^2_pu=1-pc}.
The following question seems natural to us.
 
\begin{question}\label{q:dim3_pu<1-pc}
For simply connected, one-ended, homogeneous Riemannian manifolds of dimension $d\geq3$, is \eqref{eq:p_u leq 1-p_c} always strict?
\end{question}

Next, we discuss the problem of existence of a non-uniqueness phase. For a transitive graph $G$, it is well known that $p_c(G)=p_u(G)$ if $G$ is amenable, see, e.g., \cite[Theorem 7.6]{LP16}. The famous \cite[Conjecture 6]{BS96} states that the converse is true, i.e., if $G$ is a transitive non-amenable graph, then $p_c(G)<p_u(G)$ -- see ~\cite{PS00,Hut19,Hut20} for some results in this direction, and \cite{Tyk07,Dic24} for related works in the continuum setting. We remark that this conjecture is not true in the context of unimodular random graphs, see \cite{MR3880008}. Below, $h(M)$ denotes the Cheeger constant of $M$ (see, e.g., \cite[p.~361]{Chavel} for a definition).

\begin{question}\label{q:pc<pu}
    If $M$ is homogeneous, is $p_c(M,\lambda)=p_u(M,\lambda)$ equivalent to $h(M)=0$?
\end{question}

To conclude, we remark that the behavior of critical percolation on homogeneous manifolds with a positive Cheeger constant may be more approachable than on other spaces. For instance, if $M$ is assumed to be a symmetric space with $h(M)>0$, the results of \cite{Paq17} imply that $\pazocal{G}_{M}(\eta_\lambda)$ is a non-amenable unimodular network. Therefore, it follows from \cite[Theorem 8.11]{aldouslyons} that the phase transition is continuous in the sense that $\pazocal{N}_{M}(\lambda,p_c(M,\lambda))=0$ almost surely for every $\lambda>0$. As mentioned earlier in this introduction, the corresponding result for $\R^d$, $d\geq3$, is a major challenge in percolation theory. However, similarly to transitive graphs, it is natural to expect that continuity of phase transition holds for all homogeneous manifolds. Beyond the continuity, it is also natural to expect that the critical model has a conformally invariant scaling limit at criticality, see for instance \cite{BS98,Ben15} for more on this matter.

\paragraph{Organization of the paper.} In Section~\ref{sec:discretization} we study geometric properties of graphs embedded in $M$ and prove Theorems~\ref{thm: connected cutsets} and \ref{thm:bound on p_u for Voronoi}. In Section~\ref{sec:localcomparison} we show that for $\lambda$ large, the Voronoi graph on a small neighborhood of $M$ looks like its Euclidean counterpart -- this is the content of Proposition~\ref{prop:comparison}. In Section~\ref{sec:fine_graining} we prove Theorem~\ref{thm:main}, which relies on Theorem~\ref{thm:discretization_percolation} and Proposition~\ref{prop:comparison}.
In Appendix~\ref{app:normal coordinates}, we provide basic background on Riemannian normal coordinates, which are used in Section~\ref{sec:localcomparison}. In Appendix~\ref{app:bounded cells}, we show that Poisson--Voronoi cells are bounded almost surely. In Appendix~\ref{app:phase-transition}, we show that the phase transition is non-trivial for all $\lambda>0$ under the assumptions of Theorem~\ref{thm:main} and superlinear growth.

\paragraph{Acknowledgments.} We would like to thank Gábor Pete for bringing to our attention the literature on unimodular random graphs and Alexander Lytchak for helpful comments on normal coordinates. We are also grateful to Pierre Calka, Daniel Hug, Sébastien Martineau and Vincent Tassion for their comments on earlier drafts of this paper. 
TB has been supported by the German Research Foundation (DFG) through the Priority Programme “Random Geometric Systems” (SPP 2265), via the research grant HU 1874/5-1.
The three last authors were supported by the ERC grant CriSP (No. 851565). FS was also supported by the ERC grant Vortex (No. 101043450).

\section{Connected cutsets and discretization}
\label{sec:discretization}

We start by proving a general result about the connectedness of minimal cutsets for embedded graphs (see Theorem \ref{thm: connected cutsets}) and deduce a bound on $p_u$ as a consequence (see Theorem \ref{thm:bound on p_u}). We then consider graphs induced by $\ep$-nets of $M$ and prove that critical parameters for Bernoulli site percolation on these graphs are uniformly bounded away from $0$ and $1$.

\subsection{Connectedness of minimal cutsets}

In this section we prove Theorem \ref{thm: connected cutsets}. The proof is purely based on topological facts. Namely, we rely on the fact that if an open set $U\subset M$ is such that both $U$ and its complement $M\setminus U$ are connected, then its boundary $\partial U$ is also connected. The difficulty lies in showing that such a property in the continuum persists in the discrete. 

Fix an embedded graph $G=(V,E)$ and let $(O(x))_{x\in V}$ be the family that satisfies the requirements in the statement of Theorem \ref{thm: connected cutsets}.
Recall that a continuous path in $M$ is a continuous function $\gamma:I \rightarrow M$, for some interval $I\subset \R$. We may sometimes abuse the notation and just write $\gamma$ to denote its image $\im(\gamma)$.
We understand a \emph{(discrete) path} in a graph \(G\) to be a possibly infinite sequence of vertices, so that consecutive vertices are neighbors in \(G\).
We call such a path \emph{simple} when no vertex appears more than once in it.
We define $B_{M}(x,r)$ as the open ball centered at $x$ of radius $r$ in $M$.

We collect three topological facts for later reference:
\begin{equation}\label{fact:connected_boundary}\tag{T1}
    \textit{\parbox{.8\textwidth}{Assume that \(M\) is simply connected. If \(U \subset M\) is open and connected, then \(\partial U\) is connected if and only if \(M \setminus U\) is connected.}}
\end{equation}
This is shown in \cite{CKL11} for subsets of \(\R^d\),
but the proof applies mutatis mutandis to any simply connected, normal and locally path-connected topological space.
\begin{equation}\label{fact:connected_components}\tag{T2}
    \textit{\parbox{.8\textwidth}{If \(A \subset M\) is connected and \(C\) is a connected component of \(M \setminus A\), then \(M \setminus C\) is connected.}}
\end{equation}
This follows from \cite[Theorem V.III.5]{Kur68} and the fact that \(M\) is a connected space.
\begin{equation}\label{fact:boundary_contained}\tag{T3}
    \textit{\parbox{.8\textwidth}{If \(A \subset M\) and \(C\) is a connected component of \(M \setminus A\), then \(\partial C \subset \partial A\).}}
\end{equation}
\begin{proof}[Proof of \eqref{fact:boundary_contained}]
Assume otherwise and take \(x \in \partial C \setminus \partial A\).
Obviously \(x\) does not lie in the interior of \(A\), so \(x \in M \setminus \overline{A}\).
Thus, \(B_M(x,\ep) \subset M \setminus \overline{A}\) for some \(\ep > 0\).
Since this ball is connected and intersects \(C\), it must be contained in \(C\), hence \(x\) is an interior point of \(C\), a contradiction.
\end{proof}

We now prove that any connected subset of $M$ naturally induces a connected set in $G$. In particular, this implies that $G$ itself is connected.

\begin{lemma}
\label{lem:connected_sets}
Let \(\Gamma \subset V\) and let \(A \subset \cup_{x\in\Gamma}O(x)\) be a connected set.
Then
    \[ \{x \in \Gamma \colon O(x)\cap A \neq \emptyset\} \]
is a connected subset of \(G\).
In particular, if $\gamma$ is a continuous path contained in $\cup_{x\in \Gamma}O(x)$, then we can associate to $\gamma$ a discrete path (not necessarily simple) $\pi $ in $G$ contained in $\Gamma$ such that
    \[\gamma \subset \bigcup _{x\in\pi}O(x)\qquad \text{and} \qquad\forall x\in\pi \quad \gamma\cap O(x)\ne \emptyset.\] 
\end{lemma}

\begin{proof}
    In order to show that \(S \coloneqq \{x \in \Gamma \colon O(x)\cap A \neq \emptyset\}\) is connected, we define an equivalence relation on \(S\) as follows:
    \(x \sim y\) if and only if there is a discrete path from \(x\) to \(y\) in \(G|_S\).
    Assume that there are at least two equivalence classes.
    For an arbitrary \(x \in S\), it then holds that
    \[ U \coloneqq \bigcup_{y \sim x} O(y) \qquad \text{and} \qquad W \coloneqq \bigcup_{y\not\sim x} O(y) \]
    are disjoint open sets with \(A \subset U \cup W\), \(A \cap U \neq \emptyset\) and \(A \cap W \neq \emptyset\).
    This contradicts the fact that \(A\) is connected.
    It thus follows that there is only one equivalence class, hence any two points in \(S\) can be connected by a discrete path in \(S\).

    A continuous path \(\gamma\) is in particular a connected set and hence $S = \{x \in \Gamma \colon \gamma \cap O(x) \neq \emptyset\}$ is connected as a subset of \(G\) by the previous argument.
    Then take \(\pi\) to be a path in \(S\) that traverses all vertices of \(S\).
\end{proof}

The following is an easy consequence of the lemma above, which justifies the claim in parenthesis in Theorem~\ref{thm: connected cutsets}. 

\begin{lemma} \label{lem:one_end} Assume that $M$ is one ended, the sets $O(x)$ are all bounded and any bounded set \(B \subset M\) intersects at most finitely many of the sets \(O(x)\).
Then the graph $G$ is one-ended and locally finite.
\end{lemma}
 
\begin{proof}
We start by showing that for $B \subset M$ bounded, not only does $M\setminus B$ contain a unique unbounded connected component $C$, but further $M\setminus C$ is bounded.
Without loss of generality, we can assume that $B$ is open and connected (otherwise, replace $B$ with an open ball that contains $B$).
Notice that $M \setminus \overline{B}$ contains a unique connected component $C'$ and that $C' \subset C$.
By \eqref{fact:boundary_contained}, $\partial C \subset \partial B$ is bounded, so $M \setminus \partial C$ has exactly one unbounded connected component, which clearly must contain $C'$.
By \eqref{fact:connected_components}, $M \setminus C$ is connected.
Since it is contained in a different connected component of $M\setminus \partial C$ than $C'$, it must be bounded.

Let \(\Gamma \subset V\) be finite.
Take $C$ to be the unique unbounded connected component of $M\setminus {\bigcup_{v\in \Gamma} O(v)}$.
By Lemma~\ref{lem:connected_sets}, any two points \(x,y \in V \cap C\) can be joined by a discrete path in $G$ that does not use any vertices of $\Gamma$, 
hence $V \cap C$ is connected in $G \setminus \Gamma$.
By the considerations above, $M\setminus C$ is bounded,
hence it intersects only finitely many of the sets \(O(x)\), \(x \in V\).
It follows that almost all \(x \in V\) are contained in \(C\).

For any \(x_0 \in V\) it follows from the assumed properties of the family \((O(x))_{x \in V}\) that \(\deg_G(x_0) = |\{x \in V \colon O(x) \cap O(x_0) \neq \emptyset\}| < \infty\), hence \(G\) is locally finite.
\end{proof}

We are now ready to prove Theorem~\ref{thm: connected cutsets}. This result will be useful to prove that graphs induced by $\ep$-nets have uniformly bounded critical parameter, see Theorem~\ref{thm:discretization_percolation} below. The proof is purely based on topological arguments.

\begin{proof}[Proof of Theorem \ref{thm: connected cutsets}]
    Let \(x,y \in G\) and let \(\Gamma\) be a minimal cutset separating them.
    Assume that \(\Gamma\) is not connected, then it can be decomposed into (possibly infinitely many) connected components
    \(\Gamma_1,\Gamma_2,\ldots\), to each of which we associate a set
    \( \widetilde \Gamma_i \coloneqq \bigcup_{v \in \Gamma_i}O(v) \subset M\).
    We further let \(\widetilde \Gamma \coloneqq \bigcup_{v \in \Gamma}O(v)\).
    Notice that the sets \(\widetilde \Gamma_i\) are open and connected, and that they are pairwise disjoint, i.e.,
    \(\widetilde \Gamma_i \cap \widetilde \Gamma_j = \emptyset\) for \(i\neq j\).

    Let \(C\) denote the connected component of \(\widetilde \Gamma_2\) in \(M \setminus \widetilde \Gamma_1\).
    We can assume that \(y\) is contained in \(C\).
    Otherwise, \(y\) and \(\widetilde \Gamma_1\) are both contained in \(M \setminus C\), which is a connected set by \eqref{fact:connected_components}.
    Since \(M \setminus C \subset M \setminus \widetilde \Gamma_2\), it then follows that \(y\) and \(\widetilde \Gamma_1\) are in the same connected
    component of \(M \setminus \widetilde \Gamma_2\).
    In this case, we simply switch the roles of \(\widetilde \Gamma_1\) and \(\widetilde \Gamma_2\).

    Since \(\widetilde \Gamma_1\) is open and connected, \eqref{fact:connected_components} and basic topology yield that both \(C\) and \(M \setminus C\)
    are connected, and that \(C\) is closed.
    It then follows by \eqref{fact:connected_boundary} that \(\partial C = \partial (M \setminus C)\) is connected.
    By \eqref{fact:boundary_contained}, \(\partial C\) is contained in \(\partial \widetilde \Gamma_1\).
    Since the sets \(\widetilde \Gamma_i\) are open and pairwise disjoint, it follows that \(\partial C \cap \widetilde \Gamma = \emptyset\).
    Since the sets \(O(x)\) cover \(M\), we get \(\partial C \subset \bigcup_{v \in V \setminus \Gamma}O(v)\), and since \(\partial C\) is connected, Lemma \ref{lem:connected_sets} implies that
    for any pair of points $v_1,v_2$ such that $O(v_i)\cap\partial C\neq\emptyset$, $i\in\{1,2\}$, there exists a discrete path in \(V \setminus \Gamma\) connecting $v_1$ and $v_2$. 

    We now claim that \(x \notin C\).
    Otherwise, take some \(w \in \Gamma_1\), then \(x,y \in C\) and \(O(w) \cap C = \emptyset\).
    Since \(\Gamma\) is a minimal cutset, there is a discrete path \(\pi\) from \(x\) to \(y\) with \(\Gamma \cap \pi = \{w\}\).
    By the above properties of \(x,y\) and \(w\), the first vertex \(v_1\) of \(\pi\) with \(O(v_1) \cap \partial C \neq \emptyset\) is encountered
    strictly before \(w\), and the last vertex \(v_2\) of \(\pi\) with \(O(v_1) \cap \partial C \neq \emptyset\) is encountered strictly after \(w\).
    By the previous paragraph, we can replace the subpath \((v_1,\ldots,w,\ldots,v_2)\) of \(\pi\) by a path \((v_1,\ldots,v_2)\) in \(V \setminus \Gamma\), which yields a
    modification \(\pi'\) that connects \(x\) to \(y\) and avoids \(\Gamma\) altogether, contradicting the fact that \(\Gamma\) is a cutset.

\begin{figure}
    \centering
    \includegraphics[width=0.6\linewidth]{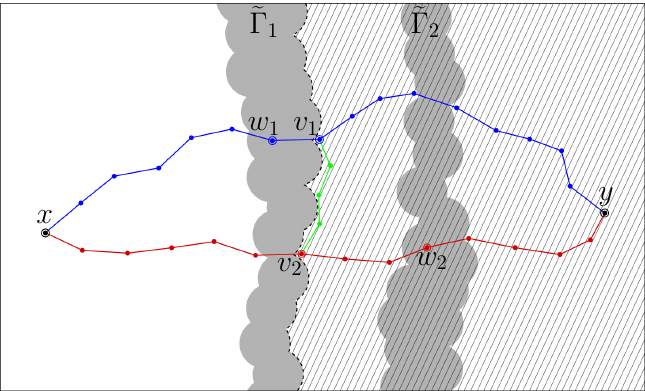}
    \caption{Picture for the proof of Theorem~\ref{thm: connected cutsets}.
    The sets \(\widetilde\Gamma_1\) and \(\widetilde\Gamma_2\) are colored gray,
    the paths \(\pi_1\) and \(\pi_2\) are drawn in blue and red respectively.
    The hatched region is the the set \(C\), its boundary is highlighted by a dashed line.
    By walking along the red path from \(x\) to \(v_2\), then along the green path to \(v_1\) and finally to \(y\) along the blue path, one avoids all vertices of \(\Gamma\).}
    \label{fig:cutset_proof}
\end{figure}

    Choose \(w_1 \in \Gamma_1\) and \(w_2 \in \Gamma_2\).
    Since \(\Gamma\) is a minimal cutset, there exist discrete paths \(\pi_1\) and \(\pi_2\) from \(x\) to \(y\) such that \(\pi_i \cap \Gamma = \{w_i\}\), \(i\in\{1,2\}\).
    It holds that \(O(w_2)\subset \widetilde \Gamma_2\) is contained in the interior of \(C\) and, as observed above, \(x \notin C\).
    Hence, the first vertex \(v_2\) of \(\pi_2\) with \(O(v_2) \cap \partial C \neq \emptyset\) is encountered strictly before \(w_2\).
    Since \(w_1 \in \Gamma_1\) and \(C \subset M \setminus \widetilde \Gamma_1\), it follows that \(O(w_1) \subset M \setminus C\).
    Recall that \(y \in C\).
    It the follows that the last vertex \(v_1\) of \(\pi_1\) with \(O(v_1) \cap \partial C \neq \emptyset\) is encountered strictly after \(w_1\).
    By the property proved two paragraphs above, there exists a discrete path \(\pi' = (v_2,\ldots,v_1)\) in \(V \setminus \Gamma\).
    Now, concatenating the subpath \((x,\ldots,v_2)\) of \(\pi_2\), the path \(\pi'\) and the subpath \((v_1,\ldots,y)\) of \(\pi_1\) yields a path in
    \(V \setminus \Gamma\) that connects \(x\) and \(y\) (cf.~Figure \ref{fig:cutset_proof}).
    This contradicts the fact that \(\Gamma\) is a cutset.
    The first part of the theorem follows.

    Now assume that $G$ is one-ended and locally finite.
    Any (necessarily finite) minimal cutset $\Gamma$ separating $x$ from $\infty$ is then also a minimal cutset separating $x$ from any vertex $y$ in the unique infinite connected component of $V\setminus\Gamma$.
    By the first part of the claim, we conclude that $\Gamma$ is connected.
\end{proof}

The following version of Theorem~\ref{thm: connected cutsets} follows easily.

\begin{corollary}\label{cor:connected cutsets closed}
    Theorem~\ref{thm: connected cutsets} holds if the family of open sets $(O(x))_{x\in V}$ is replaced by a family of closed sets $(C(x))_{x\in V}$ satisfying the same conditions in addition to the following locally finiteness property: for every bounded subset $B\subset M$, $|\{x\in V:\, C(x)\cap B \neq \emptyset\}|<\infty$.
\end{corollary}

\begin{proof}
    For every $x\in V$ and $z\in C(x)$, our assumptions on $(C(x))_{x \in V}$ guarantee that the set $S_x(z)\coloneqq\{y\in V:\, C(y)\cap C(x)=\emptyset,\, C(y)\cap B_M(z,1)\neq \emptyset\}$ is finite.
    Hence, the radius $r_x(z)\coloneqq \min\{1,\dm\big(z,\bigcup_{S_x(z)} C(y)\big)\}/2$ is strictly positive.
    Consider the open sets $O(x)\coloneqq \bigcup_{z\in C(x)} B_M(z,r_x(z))$, and note that by construction $C(x)\cap C(y)\neq \emptyset$ if and only if $O(x)\cap O(y)\neq \emptyset$, so we can simply apply Theorem~\ref{thm: connected cutsets}.
\end{proof}

\subsection{Bound on \texorpdfstring{$\overline{p}_u$}{\textbackslash bar\{p\}\_u} for graphs with connected minimal cutsets}
\label{sec:Peierls_G_ep}

Let $G=(V,E)$ be a graph and $\P_p$ denote the law of Bernoulli site percolation on $G$ with parameter $p$. We use the terminology cluster to refer to an open connected component.
Define 
$$p_c(G)\coloneqq\sup\{\,p\in[0,1]: \P_p(\text{there exists an infinite cluster in $G$})=0\},$$
$$p_u(G)\coloneqq \sup\{\,p\in[0,1]: \P_p(\text{there exists a unique infinite cluster in $G$})<1\}$$
and 
$$\overline p_u(G)\coloneqq \sup\Big\{\,p\in[0,1]: \P_p\left(\begin{array}{c}\text{there exists an open infinite cluster $C$ in $G$} \\\text{such that $G\setminus C$ does not contain }\\ \text{an infinite connected component}\end{array}\right)<1\Big\}.$$

\begin{theorem}\label{thm:bound on p_u}Let $G=(V,E)$ be a one-ended graph that has connected minimal cutsets between any pair of vertices. 
Then, we have
\[p_u(G)\le \overline p_u(G)\le \max\{p_c(G), 1-p_c(G)\}.\]
\end{theorem}

\begin{proof}
The proof follows almost the same lines as those of \cite[Theorem 10]{BB99}. 
Fix any $p>\max\{p_c(G), 1-p_c(G)\}$. We will prove that if an infinite cluster $C$ is such that $G\setminus C$ has an infinite connected component, then one can find an infinite connected set of closed vertices, thus contradicting the fact that $1-p<p_c(G)$.
Indeed, assume that such a cluster $C$ exists and let $A$ be an infinite connected component of $G\setminus C$.
Define the interior boundary of \(A\) as
\[\partial^\text{int}A \coloneqq \{x\in A :\exists y\in C\quad \{x,y\}\in E\}.\]
All vertices of \(\partial^\text{int}A\) are closed and one easily checks that since \(G\) is one-ended, \(\partial^\text{int}A\) is infinite or some vertex of \(\partial^\text{int}A\) has infinitely many neighbors.
(In the latter case, it holds that \(p_c(G)=0\) and the claim follows immediately, so we assume the former occurs.)
We now show that \(\partial^\text{int}A\) is connected, thus concluding the proof.
Assume otherwise, then there exist paths \(\pi\) in \(A\) with endpoints \(x,y \in \partial^\text{int}A\) such that there is no path \(\wt \pi\) contained in \(\partial^\text{int}A\) with the same endpoints.
Among those paths, choose a \(\pi\) with minimal length.
Clearly, it is of the form \((x,z_1,\ldots,z_m,y)\) with \(m\geq 1\) and \(z_1,\ldots,z_m \in A\setminus \partial^\text{int}A\).
Since \(\partial^\text{int}A\) separates \(z_1\) from any base point \(v_0\in C\), we can extract a minimal cutset \(D \subset \partial^\text{int} A\), which is connected by the conditions of the theorem.
By construction, \(D\) must contain \(x\) and \(y\), so these vertices are connected in \(D\) and hence in \(\partial^\text{int}A\), a contradiction. 
\end{proof}

\begin{proof}[Proof of Theorem~\ref{thm:bound on p_u for Voronoi}]
     Since $\eta_\lambda$ is locally finite almost surely, one can easily verify that Corollary~\ref{cor:connected cutsets closed} applies to the Voronoi graph $\pazocal{G}_M(\eta_\lambda)$, so it has connected minimal cutsets. By Theorem~\ref{thm:bdd-cells_continuous} and Remark~\ref{rem:bdd-cells_curvature}, all the cells of $\pazocal{G}_M(\eta_\lambda)$ are compact almost surely.
     Lemma~\ref{lem:one_end} (applied to the modified family $(O(x))_{x\in\eta_\lambda}$ constructed in the proof of Corollary \ref{cor:connected cutsets closed}) then yields that $\pazocal{G}_M(\eta_\lambda)$ is one-ended.
     The result follows readily from Theorem~\ref{thm:bound on p_u}.
\end{proof}

\subsection{Discretization of \texorpdfstring{$M$}{M}}\label{sec:geometry_G_ep}

In this section, we construct a discretization of $M$ by $\ep$-nets and prove that percolation on the associated graph has uniformly bounded critical parameters under the assumptions of Theorem~\ref{thm:main}.

For each $\varepsilon>0$ let  $V_\varepsilon $ be a countable subset of $M$ satisfying the following properties. 
\begin{enumerate}[(P1)]
    \item \label{property:1}  $B_M(x,\varepsilon)\cap B_M(y,\varepsilon)=\emptyset $ for all $x,y\in V_\ep$, 
    \item \label{property:2} $\bigcup_{x\in V_\ep}B_M(x,2\ep)=M$.
\end{enumerate}
To see that such a set exists, simply consider (say, by Zorn's lemma) a maximal set (with respect to inclusion) of points within distance at least $2\ep$ from each other.

We then define the graph 
$G_\varepsilon=(V_\varepsilon, E_\varepsilon)$ with edge set
\begin{equation}\label{eq:edges_G_ep} 
E_\ep \coloneq \{\{x,y\}\subset V_\ep:~ B_M(x,2\ep)\cap B_M(y,2\ep)\ne \emptyset,\, x \neq y\}.
\end{equation}

The main result of this section says that $p_c(G_\ep)$, $p_u(G_\ep)$ and $\overline p_u(G_\ep)$ are uniformly bounded away from $0$ and $1$ for small $\ep>0$.

\begin{theorem}\label{thm:discretization_percolation} Assume $M$ satisfies conditions \ref{thm cond: global injectivity radius}--\ref{thm cond: log expansion} of Theorem \ref{thm:main}.
Then, for every $\varepsilon_0>0$, there exists a constant $\delta_0=\delta_0(M,\varepsilon_0)>0$ such that 
\begin{equation}
    \delta_0\leq p_c(G_\varepsilon)\leq p_u(G_\varepsilon)\leq \overline p_u(G_\ep)\le 1-\delta_0 ~~~\forall \varepsilon\in(0,\varepsilon_0].
\end{equation}
\end{theorem}

\begin{remark}
We in fact prove the following more quantitative version of Theorem \ref{thm:discretization_percolation} for all $\ep>0$:
\begin{equation}\label{eq:perco_bounds_G_ep}
\frac 1 {\Delta(\ep)}\leq p_c(G_\varepsilon)\leq p_u(G_\varepsilon)\leq \overline p_u(G_\ep)\le 1-\frac {1}{e^{C_1\ep}\Delta(\ep)},
\end{equation}
where $\Delta(\ep)$ is defined in \eqref{eq:def_Delta} and $C_1$ is a constant that depends only on $M$.
\end{remark}

The uniform lower bound on $p_c$ follows easily from the fact that $G_\ep$ has uniformly bounded degrees for $\ep\in(0,\ep_0]$, which can be proved thanks to the uniform bound on the curvature \ref{thm cond: uniform bound on curvature} -- see Lemma~\ref{lem:bounded_degree}. For the upper bound on $\overline p_u$, we use Theorem \ref{thm: connected cutsets} (with $O(x)=B_M(x,2\ep)$) to deduce that minimal cutsets are connected for $G_\ep$. Thanks to Theorem \ref{thm:bound on p_u}, a uniform upper bound on $p_c$ is sufficient to conclude. To obtain such a bound, we prove that any connected cutset separating a large ball from infinity contains a large enough path -- see Lemma~\ref{lem:cutsets->paths}.

First, we prove that the graphs $G_\ep$ have uniformly bounded degrees. For $\ep>0$, define
\begin{equation}\label{eq:def_Delta}
 \Delta(\ep) \coloneqq \sup_{x \in M} N(B_M(x,5\ep),\ep), 
\end{equation}
    where the \emph{packing number} \(N(B_M(x,5\ep),\ep)\) is the maximal number of disjoint balls of radius \(\ep\) that fit into \(B_M(x,5\ep)\).

\begin{lemma}\label{lem:bounded_degree}
   Every $x\in V_\ep$ has at most $\Delta(\ep)$ neighbors.
\end{lemma}

\begin{proof}
    Note that if $y,x\in V_\ep$ are neighbors, then $\dm(x,y)\le 4\ep$ by \eqref{eq:edges_G_ep}, and therefore $B_M(y,\ep)\subset B_M(x,5\ep)$.  By Property \ref{property:1} of $V_\ep$, we see that the number of neighbors of $x$ is at most $N(B_M(x,5\ep),\ep) \leq \Delta(\ep)$, as we wanted to prove.
\end{proof}

\begin{lemma}\label{lem: bound on Delta} Assume $M$ is a complete manifold satisfying conditions \ref{thm cond: global injectivity radius} and \ref{thm cond: uniform bound on curvature} of Theorem \ref{thm:main}, then it holds for all \(\varepsilon_0>0\) that \[\sup_{\ep\in(0,\ep_0]}\Delta(\ep)<\infty.\]
\end{lemma}

\begin{proof}
    For \(c,r \in \R\), let \(V_c(r)\) denote the volume of a ball of radius \(r\) in a \(d\)-dimensional space form of constant sectional curvature \(c\).
    More precisely, \(V_c(r) = \omega_{d} \int_0^r s_c(t)^{d-1} dt\), where \(\omega_d \coloneqq 2 \pi^{d/2}/\Gamma(d/2)\) is the surface area of the \(d\)-dimensional euclidean unit ball and
    \begin{equation}\label{eq:def_sc}
        s_c(r) = \begin{cases}
        r,              & c=0,\\
        R \sin(r/R),    & c = 1/R^2 > 0,\\
        R \sinh(r/R),   & c = -1/R^2 < 0,
    \end{cases}
    \end{equation}
    for \(r \in \R\) (cf.~\cite[Equations (III.4.1), (III.3.10) and (II.5.8)]{Chavel}).
    Choose \(c>0\) such that all sectional curvatures of \(M\) lie in \([-c,c]\).
    Then \cite[Equation (III.4.21)]{Chavel} implies that
   \begin{equation}\label{eq:upper_bound_volume}
       \mu_M(B_M(x,r)) \leq V_{-c}(r) 
   \end{equation}
    for all \(x \in M\) and \(r>0\).
    Further, it holds by \cite[Equation (III.4.10)]{Chavel} that
    \begin{equation}\label{eq:lower_bound_volume}
        \mu_M(B_M(x,r)) \geq V_c(r)
    \end{equation}
    for all \(x\in M\) and \(r \leq \min\{r_\text{inj}(M),\pi/\sqrt{c}\}\).

    For any \(\ep>0\) it then holds that
    \[ \Delta(\ep) \leq \frac{\sup_{x\in M}\mu_M(B_M(x,5\ep))}{\inf_{y\in M}\mu_M(B_M(y,\ep))} \leq \frac{V_{-c}(5\ep)}{V_c(\min\{\ep,r_\text{inj}(M),\pi/\sqrt{c}\})}. \]
    The last term is continuous in \(\ep\) and one easily checks that it goes to \(5^d\) for \(\ep \to 0+\).
    The claim then follows directly.
\end{proof}

It turns out that the same constant $\Delta(\ep)$ given by \eqref{eq:def_Delta} can be used to compare the graph distance $\d_{G_\ep}$ in $G_\ep$ with the geodesic distance $\dm$.

\begin{lemma}\label{lem:quasi_isometry}
    For every $\ep>0$ and every distinct vertices $x,y\in V_\ep$,
    \begin{equation}
        \frac{1}{4\varepsilon} \dm(x,y)\leq \d_{G_\ep}(x,y) \leq \frac{\Delta(\ep)}{2\varepsilon} \dm(x,y).
    \end{equation}
\end{lemma}

\begin{proof}
    By definition, if $x$ and $y$ are neighbors in $G_\ep$, then $\dm(x,y)\leq 4\ep$. By considering a shortest path between $x$ and $y$ in $G_\ep$ and applying the triangle inequality, it follows that $\dm(x,y)\leq 4\ep \d_{G_\ep}(x,y)$. We now prove the second inequality. Fix two distinct vertices $x,y\in G_\ep$ and let $\gamma=(\gamma_t)_{t\in[0,\ell]}$ be a minimizing geodesic\footnote{Since $M$ is connected and complete, it holds that for any two points $x,y \in M$ there is a geodesic of minimal length $d_M(x,y)$ connecting $x$ and $y$ (see, e.g., \cite[Corollary 6.21]{LeeRM}). We call this curve a \emph{minimizing geodesic}. Note that it may not be unique.} in $M$ between $x$ and $y$, parametrized so that $\dm(\gamma_s,\gamma_t)=|t-s|$ and $\ell=\dm(x,y)$. 
    By Lemma~\ref{lem:connected_sets}, we can find a simple path $x=z_0,z_1,\ldots,z_k=y$ in $G_\ep$ with $z_i\in V_\ep \cap \{y\in M:~\dm(y,\gamma)<2\ep\}$ for every $i$.
    In particular, 
    $$B_M(z_i,\varepsilon)\subset \{y\in M:~\dm(y,\gamma) < 3\ep\}\subset \bigcup_{j\in\{1,\ldots,\lfloor \ell/2\ep \rfloor\}} B_M(\gamma_{2\ep j},5\ep).$$ 
    Therefore, by Property \ref{property:1} of $V_\ep$, we have \[k\leq \sum_{j\in\{1,\ldots,\lfloor \ell/2\ep \rfloor\}} N(B_M(\gamma_{2\ep j},5\ep),\ep)\leq \frac{\Delta(\ep)}{2\varepsilon} \ell,\] and the desired inequality $\d_{G_\ep}(x,y) \leq \frac{\Delta(\ep)}{2\varepsilon} \dm(x,y)$ follows.
\end{proof}

The following lemma will be used to count the number of possible minimal cutsets separating $x$ from infinity, which is crucial to prove the uniform upper bound on $p_c$. Let $\mathcal{P}_\ep (v,n)$ be the set of all simple paths of length $n$ starting from a vertex $v\in G_\ep$.

\begin{lemma}\label{lem:cutsets->paths}
    Assume that $M$ satisfies the conditions of Theorem \ref{thm:main}. 
    Let $\ep>0$ and  $x\in V_\ep$.
    There exists a constant $C_1$ independent of $\ep$ and $x$ as well as constants $C_0=C_0(\ep,x),c_0=c_0(\ep,x)>0$  such that the following holds:
    For every $k\ge 1$ there exists a subset  $V_k^x \subset V_\ep$ such that $|V_k^x|\le C_0e^{C_1\ep k}$ and, for every $n\geq 1$, any
    connected cutset contained in $G_\ep \setminus B_{G_\ep}(x,n)$ separating $B_{G_\ep}(x,n)$ from infinity contains at least one path in
    \[\bigcup_{k\ge c_0\log n-C_0}\bigcup_{v\in V_k^x}\mathcal{P}_{\ep}(v,k).\]
\end{lemma}

\begin{proof}
    Let $\gamma=(\gamma_t)_{t\in\R}$ be a log-expanding bi-infinite path in $M$, chosen independently of \(x\) and \(\ep\).
    Recall that \(\gamma\) is Lipschitz continuous and satisfies
    \[\dm(\gamma_t,\gamma_s)\ge c\log |t-s|-C,\]
    for some constants \(c,C>0\).
    Without loss of generality we will assume \(\gamma\) to be $1$-Lipschitz continuous, as this can always be achieved be reparametrizing.

    Now fix \(\ep > 0\) and \(x \in V_\ep\).
    At the expense of changing the constant \(C\) (while leaving \(c\) unchanged) we can modify \(\gamma\) in such a way that it passes through \(x\) while preserving $1$-Lipschitz continuity and the above inequality.
    For convenience, we will assume that \(\gamma_0 = x\).
    
    We claim that there exists a bi-infinite path $\pi=(\pi_i)_{i\in\Z}$ in $G_\ep$ such that $\pi_0=x$ and 
    \begin{equation}\label{eq:quasi_bigeodesic}
        \d_{G_\ep}(\pi_i,\pi_j)\geq \frac{c}{4\ep}\log |i-j|  - C(\ep,x)
    \end{equation}
    for every $i,j\in\Z$, where $C(\ep,x)>0$ is a constant depending only on $\ep$ and $x$. Indeed, we will construct such a path in a similar way as in the proof of Lemma~\ref{lem:quasi_isometry}. Define the times $t_k\coloneq 4\ep k$, $k\in\Z$. We can apply Lemma~\ref{lem:connected_sets} to each subpath $\gamma_{[t_k,t_{k+1}]}\coloneq (\gamma_t)_{t\in[t_k,t_{k+1}]}$ in order to obtain a discrete simple path $\pi^{(k)}=(\pi^{(k)}_1,\ldots,\pi^{(k)}_{n_k})$ such that $\gamma_{t_k}\in B_M(\pi^{(k)}_1,2\ep)$, $\gamma_{t_{k+1}}\in B_M(\pi^{(k)}_{n_k},2\ep)$ and $\d(\pi^{(k)}_i,\gamma_{[t_k,t_{k+1}]})<2\ep$ for every $i\in \{1,\ldots,n_k\}$. We can further assume that $\pi^{(0)}_1=x$. Now, define $\pi$ as the concatenation of the paths $\pi^{(k)}$, $k\in\Z$, in such a way that $\pi_0=\pi^{(0)}_1=x$.
    Note that \(\pi\) is not necessarily a simple path.
    
    We will now prove that $\pi$ constructed above indeed satisfies \eqref{eq:quasi_bigeodesic}. Since $\gamma$ is $1$-Lipschitz continuous, we have for any $t\in[t_k,t_{k+1}]$ that  $\dm(\gamma_{s_k},\gamma_{t})\le 2\ep$, where $s_k\coloneqq \frac{t_k + t_{k+1}}{2}$.  It follows that $B_M(\pi^{(k)}_i,\ep)\subset B_M(\gamma_{s_k},5\ep)$ for every $i\in \{1,\ldots,n_k\}$. In particular, due to Property~\ref{property:1} of $V_\ep$, we have $n_k\leq \Delta(\ep)$ for every $k\in\Z$. Now, given $i\neq j$, let $s\geq 0$ be such that $|i-j|\in (\Delta(\ep) s, \Delta(\ep)(s+1)]$, and $m,\ell \in \Z$ be such that $\pi_i\in \pi^{(m)}$ and $\pi_j\in \pi^{(\ell)}$. Since $n_k\leq \Delta(\ep)$ for every $k\in\Z$, it follows that $|m-\ell|\geq s$. Now, by Lemma~\ref{lem:quasi_isometry} and the triangle inequality,
    \begin{align}
        \d_{G_\ep}(\pi_i,\pi_j)\geq (4\ep)^{-1} \dm(\pi_i,\pi_j)&\geq (4\ep)^{-1} \left(\dm(\gamma_{s_m},\gamma_{s_\ell}) - \dm(\pi_i,\gamma_{s_m}) - \dm(\pi_j,\gamma_{s_\ell}) \right) \\ 
        &\geq (4\ep)^{-1} \left(c \log |m-\ell| +c\log(4\ep)- C- 4\ep - 4\ep \right)\\ 
        &\geq (4\ep)^{-1}(c\log s +c\log(4\ep)-C) -2 \label{eq:constbound}\\&\geq c(4\ep)^{-1}\log |i-j|  - C(\ep,x),   \end{align}
    for a suitably chosen \(C(\ep,x)\).
    We define 
    $$V^x_k\coloneqq \{\pi_i:~0\leq i\leq \exp(4c^{-1}\ep(k+C(\ep,x)))\}$$ 
    and note that $|V_k^x|\leq C_0(\ep,x) e^{C_1 k\ep}$ for $C_0(\ep,x)$ large enough depending on $\ep$ and \(C_1 = 4c^{-1}\).
     
    Let $\Gamma$ be a connected cutset separating $B_{G_\ep}(x,n)$ from infinity, and let $k=|\Gamma|$. Therefore, it only remains to prove that $k\geq c_0 \log n- C_0$ and that there exists $v\in\Gamma\cap V_k^x$. Since $\Gamma$ is a cutset and both $(\pi_i)_{i\geq 0}$ and $(\pi_j)_{j\leq 0}$ are paths from $B_{G_\ep}(x,n)$ to infinity, we can find $i\geq 0$ and $j\leq 0$ such that $\pi_i,\pi_j \in\Gamma$. Since $\Gamma\cap B_{G_\ep}(x,n)=\emptyset$, we have $\Gamma\cap B_{G_\ep}(x,n-1) =\emptyset$, and in particular $i\geq n$ and $j\leq -n$. By the fact that $\Gamma$ is connected and \eqref{eq:quasi_bigeodesic}, we deduce that 
    \begin{equation}
        k=|\Gamma|\geq \d_{G_\ep}(\pi_i,\pi_j)\geq (4\ep)^{-1}c\log|i-j|-C(\ep,x).
    \end{equation}
    Since $|i-j|\geq 2n$, we deduce that $k\geq c_0 \log n-C(\ep,x)$. Since $|i-j|\geq i$, we deduce that $i\leq \exp(4c^{-1}\ep(k+C(\ep,x)))$, and in particular $\pi_i\in V_k^x$, the claim follows by considering a path contained in $\Gamma$ joining $\pi_i$ and $\pi_j$.
\end{proof}

We are now finally ready to present the proof of Theorem~\ref{thm:discretization_percolation}.

\begin{proof}[Proof of Theorem \ref{thm:discretization_percolation}] Let $\ep\in(0,\ep_0]$. We obviously have $p_c(G_\ep)\le p_u(G_\ep)\le \overline p_u(G_\ep)$ by definition. 
We will now prove that $p_c(G_\ep)\geq \Delta(\ep)^{-1}$.
By Lemma \ref{lem: bound on Delta}, this leads to a uniform lower bound for \(\ep \in (0,\ep_0]\). 
Let $\mathcal{P}_\ep (v,n)$ be the set of all simple paths of length $n$ starting from a vertex $v\in G_\ep$. Notice that $|\mathcal{P}_\ep (v,n)|\leq \Delta(\ep)^n$ by Lemma~\ref{lem:bounded_degree}. 
We write $v\leftrightarrow\infty$ for the event that there exists an infinite simple open path in $G_\ep$ starting from $v\in G_\ep$. Therefore, for any $v\in G_\ep$ and $p<\Delta(\ep)^{-1}$, a union bound gives
\begin{equation}\label{eq:subcritical_Peierls}
\P_p(v\leftrightarrow\infty)=\lim_{n\rightarrow\infty}\P_p(\bigcup_{\pi\in \mathcal{P}_\ep (v,n)} \{\pi \text{ is fully open}\} )\le\lim_{n\rightarrow\infty} (\Delta(\ep) p)^n =0.
\end{equation}
Hence, since the number of vertices in $G_\ep$ is countable, we obtain that there exists almost surely no infinite cluster when $p<\Delta(\ep) ^{-1}$. This implies the first inequality of \eqref{eq:perco_bounds_G_ep}, which is uniform over $\ep\in(0,\ep_0]$ by Lemma~\ref{lem: bound on Delta}.

By Theorem \ref{thm: connected cutsets} and Lemma \ref{lem:one_end} (with $O(x)=B_M(x,2\ep)$, $x\in V_\ep$), we have that minimal cutsets (between any two vertices as well as between a vertex and infinity) in $G_\ep$ are connected and that $G_\ep$ is one-ended.
Thanks to Theorem \ref{thm:bound on p_u}, we have $\overline p_u(G_\ep) \le \max\{p_c(G_\ep),1-p_c(G_\ep)\}$. Hence, it remains to prove a uniform upper-bound on $p_c(G_\ep)$ in order to conclude.

Let $x\in V_\ep$, $n\ge 1$ and let $V_k^x$ be the sets defined in the statement of Lemma~\ref{lem:cutsets->paths}. If the open cluster of $B_{G_\ep}(x,n)$ (the union of $B_{G_\ep}(x,n)$ with all the open clusters intersecting it) is finite, its boundary is a cutset from $x$ to infinity, from which we can extract a minimal (and hence connected) cutset $\Gamma$ from $x$ to infinity. Moreover, this set $\Gamma$ is fully closed.
By Lemma~\ref{lem:cutsets->paths} and a union bound, we have
\begin{equation*}
    \begin{split}
        \P_p(B_{G_\ep}(x,n)\not \leftrightarrow\infty )&\le \sum_{k\ge c_0\log n-C_0 }\sum_{w\in V_k^x}\sum_{\pi\in\mathcal{P}_{\ep}(w,k)}\P_p(\pi \text{ is fully closed})\\&= \sum_{k\ge c_0\log n-C_0}\sum_{w\in V_k^x}\sum_{\pi\in\mathcal{P}_{\ep}(w,k)}(1-p)^{k}
        \le \sum_{ k\ge c_0\log n-C_0}C_0 e^{C_1 k \ep} \Delta(\ep)^k (1-p)^{k} .
    \end{split}
\end{equation*}
Recall that \(c_0,C_0\) depend on \(\ep\) while \(C_1\) does not.
For \(p > 1 - e^{-C_1\ep} \Delta(\ep)^{-1}\), it follows readily from the previous calculation that for $n$ large enough depending on $\ep$ and $x$,
\begin{equation}\P_p(B_{G_\ep}(x,n)\not \leftrightarrow\infty ) < 1.
\end{equation}
In particular, this implies the second inequality in \eqref{eq:perco_bounds_G_ep}, which is uniform over $\ep\in(0,\ep_0]$ by Lemma~\ref{lem: bound on Delta}.
\end{proof}

\begin{remark}\label{rem:pc>0_minimalassumption}
    Note that in order to prove that $p_c(G_\ep)>0$, we only used that $G_\ep$ has bounded degree, which followed from the fact that $M$ has bounded curvature. Therefore, assumptions \ref{thm cond : connected+one-ended} and \ref{thm cond: log expansion} (from Theorem \ref{thm:main}) are unnecessary for this purpose.
\end{remark}

\begin{remark}\label{rem:exp_cutsets_transience}
    Note that connectedness of cutsets is not a priori necessary to prove $p_c(G_\ep)<1$. Indeed, it is enough to show that the number of minimal cutsets of a given size grows at most exponentially. Therefore, if one only wanted to bound $p_c(G_\ep)$ -- and, as a consequence, only prove that $p_c(M,\lambda)\to p_c(\R^d)$, see Section~\ref{sec:fine_graining} and Remarks~\ref{rem:sub_fine-graining_minimalassumption} and \ref{rem:sup_fine-graining_pc_only} therein -- one could replace assumptions \ref{thm cond : connected+one-ended} and \ref{thm cond: log expansion} by some geometric property of $M$ that can be shown to imply such an exponential bound on the number of minimal cutsets in $G_\ep$. In \cite{EST24}, a new approach was developed to upper bound the number of minimal cutsets even when one cannot hope for them to be connected in any sense. More precisely, such a bound is obtained by only assuming that graph is uniformly transient. Therefore, one could hope to either adapt the proof of \cite{EST24} to the continuum setting or try to directly prove that $G_\ep$ is uniformly transient, in order to show that $p_c(G_\ep)<1$ under the assumption that the Brownian motion on $M$ is (uniformly) transient (or that $M$ has isoperimetric dimension $>2$). We do not pursue this direction here, but we believe that this approach is plausible.
\end{remark}

\section{Local comparison}\label{sec:localcomparison}

In this section, we prove a technical comparison statement that allows us to directly relate, for $\lambda$ large, a Voronoi diagram on \(\R^d\) to one on \(M\) in a small neighborhood of a point.
This is the content of Proposition \ref{prop:comparison} below.

Although it will not be relevant for the proof, we want to point out that the related question of studying properties of the typical Poisson--Voronoi cell as \(\lambda \to \infty\) has been treated in \cite{calka2021poisson}.

Recall that for a locally finite set \(\omega \subset M\), $\pazocal G_M(\omega)$ denotes the graph induced by the Voronoi tessellation associated to $\omega$ in $M$.
Further, for any graph \(\pazocal G\) with vertex set \(V\) and edge set \(E\), we define the restriction \(\pazocal G |_S\) to a set \(S\)
as the graph with vertex set \(V \cap S\) and an edge between \(x,y \in V \cap S\) if and only if they are neighbors in \(\pazocal G\).

We define the \emph{diameter} of a nonempty set \(A \subset M\) as \(\diam_M(A) \coloneqq \sup_{x,y\in A} \dm(x,y) \).
For a Poisson point process on \(M\) with intensity measure \(\mu\), we introduce the shorthand notation \(\eta \sim \PPP(M,\mu)\).
While we will view a Poisson point process as a random point set most of the time, we will sometimes implicitly interpret it as a random measure when convenient.

A \emph{chart} \(\varphi\) on \(M\) is a homeomorphism \(\varphi \colon \dom(\varphi) \to W \) from an open set \(\dom(\varphi) \subset M\)
to an open set \(W \subset \R^d\).
For a chart \(\varphi\) on \(M\), the \emph{pushforward} of the distance function \(d_M\) is defined as
\[ \varphi_* d_M(x,y) \coloneqq d_M(\varphi^{-1}(x), \varphi^{-1}(y)),\quad x,y \in \im(\varphi), \]
while the pushforward of a measure \(\mu\) on \(M\) is defined in the usual way as
\[ \varphi_* \mu(A) \coloneqq \mu(\varphi^{-1}(A)), \quad A \subset \R^d \text{ measurable}. \]
We further define the pushforward \(\varphi_* \pazocal G_M(\eta)\) of a Voronoi graph on \(M\) as follows:
The vertex set is given by \(\varphi(\eta \cap \dom(\varphi))\) and two vertices \(\varphi(x)\) and \(\varphi(y)\) are neighbors
if and only if \(x\) and \(y\) are neighbors.
For a set \(A \subseteq M\) not necessarily contained in \(\dom(\varphi)\),
we introduce the convention \(\varphi(A) \coloneqq \varphi(A\cap\dom(\varphi))\).

The goal is now to find, for \(x_0 \in M\), a coupling \((\eta,\xi)\) of processes \(\eta \sim \PPP(\lambda \mu_M)\), \(\xi \sim \PPP(\lambda dx)\) and a chart
\(\varphi\) with \(x_0 \in \dom(\varphi)\) so that with high probability \(\varphi_* \pazocal G_{M}(\eta)\) and \(\pazocal G_{\R^d}(\xi)\) agree in a ball around \(\varphi(x_0)\).
Asking for the even stronger condition that not only the Voronoi graphs but also the Voronoi tessellations agree
(i.e., \(\varphi(C(y;\eta)) = C(\varphi(y); \xi)\) for \(y\in\eta\) close to \(x_0\)) is in general too much to ask for --
this is possible when \(\varphi\) is an isometry, but such charts do not exist when the curvature at \(x_0\) is non-zero.
\paragraph{\textbf{Normal coordinates.}}
We quickly review the notion of normal coordinates, the reader is referred to \cite[p.~131-3]{LeeRM} for a more detailed introduction.
For \(x \in M\) the \emph{exponential map} \(\exp_x \colon T_x M \to M\) maps \(v \in T_x M\) to \(\gamma_v(1)\) where \(\gamma_v\) is the unique maximal geodesic
through \(x\) with velocity \(v\) at \(x\) ($\exp_x$ is defined on the entire tangent space since $M$ is complete).
The \emph{injectivity radius} at \(x\), \(r_\text{inj}(x)\), is the largest \(r>0\) such that \(\exp_x\) is a diffeomorphism from \(B_{T_x M}(0,r)\) onto its image.
(When the exponential map is a diffeomorphism on the entire tangent space, the injectivity radius at \(x\) is \(\infty\).)
The injectivity radius of \(M\) is then defined as \(r_\text{inj}(M) \coloneqq \inf_{x \in M} r_\text{inj}(x)\).
Note that throughout most of the paper we require that \(r_\text{inj}(M)>0\).

Let \( r \leq r_\text{inj}\) and \(x_0 \in M\).
Take any orthonormal basis \(v_1,\ldots,v_d\) of \(T_{x_0} M\) and let \(\Phi \colon \R^d \to T_{x_0} M\) be the linear map sending \(e_i\) to \(v_i\).
Then \(\exp_{x_0} \circ \Phi |_{B_{\R^d}(0,r)}\) is a diffeomorphism from \(B_{\R^d}(0,r)\) to \(B_M({x_0},r)\) (cf.~\cite[Corollary 6.13]{LeeRM}).
Its inverse
\[ \Phi^{-1} \circ \exp_{x_0}^{-1} \colon B_M({x_0},r) \to B_{\R^d}(0,r) \]
is a \emph{normal coordinate chart at \({x_0}\)}.
We let \(\pazocal{E}({x_0},r)\) denote the collection of all such charts.
Note that the only freedom in the above definition is the choice of an orthonormal basis,
so for any \(\varphi,\widetilde\varphi \in \pazocal{E}(x_0,r)\) there is a matrix \(Q \in O(n)\) such that \(\widetilde\varphi = Q \varphi\).
Note further that when $0 < s \leq r \leq r_\text{inj}$ and $\varphi \in \pazocal{E}(x_0,r)$, then its restriction to $B_M(x_0,s)$ satisfies $\varphi|_{B_M(x_0,s)} \in \pazocal{E}(x_0,s)$.

The following lemma gives some useful properties of normal coordinates.
Its proof is deferred to the appendix.

\begin{lemma}\label{lem:normal}
    Assume that all sectional curvatures of \(M\) lie in \([-K,K]\), where \(K\) is a positive constant.
    Let \(x_0 \in M\) and \(\varphi \in \pazocal{E}(x_0,r)\) where
    \begin{equation}
    \label{eq:radius_condition}
        0 < r \leq \min\left\{\frac{r_\text{inj}(M)}{2} , \frac{\pi}{2\sqrt{K}} \right\}.
    \end{equation}
    Then the following holds:
    \begin{enumerate}[(1)]
        \item \(\varphi_* \d_M(0,x) = \de(0,x)\), for all \(x \in B_{\R^d}(0,r)\),
        \item \((1-Kr^2) \de(x,y) \leq \varphi_* \d_M(x,y) \leq (1+Kr^2) \de(x,y)\), for all \(x,y \in B_{\R^d}(0,r)\),
        \item \((1-Kr^2)^{d/2} |A| \leq \varphi_*\mu_M(A) \leq (1+Kr^2)^{d/2} |A|\), for all measurable \(A \subset B_{\R^d}(0,r)\), where \(|A|\) denotes the Lebesgue
        measure of \(A\).
    \end{enumerate}
\end{lemma}

Now, we state the main result of this section.

\begin{proposition}\label{prop:comparison}
    Assume that all sectional curvatures of \(M\) lie in \([-K,K]\) for some \(K>0\) and that \(r_\text{inj}(M) > 0\).
    Then for all $\delta>0$, there exists an $r_0=r_0(\delta)$ such that for all $r\geq r_0$ there exists a $\lambda_0 = \lambda_0(r,\delta)$ such that for all $\lambda\geq \lambda_0$ and $x_0\in M$ we can find a coupling \((\xi, \eta)\) of Poisson point processes
    \begin{equation}\label{eq:coupling_marginals}
        \xi\sim \mathrm{PPP}(\lambda\cdot dx, \R^d) \quad \text{and} \quad \eta \sim \mathrm{PPP}(\lambda\cdot \mu_M, M),
    \end{equation}
    and a chart \(\varphi \in \pazocal{E} (x_0,r\lambda^{-1/d})\) so that the following holds with probability at least \(1-\delta\):
    \begin{enumerate}[(1)]
        \item \label{property 1 coupling} $ \pazocal G_{\R^d}(\xi)|_S = \varphi_*\pazocal G_{M}(\eta), \quad \text{where } S=B_{\R^d}\big(0, r\lambda^{-1/d}\big)$,
        \item \label{property 2 coupling} if a cell \(C(x;\xi)\), \(x\in\xi\), intersects \(S\), then \(\diam_{\R^d}(C(x;\xi)) \leq r\lambda^{-1/d}/100\),
        \item \label{property 3 coupling} if a cell \(C(x;\eta)\), \(x\in\eta\), intersects \(\varphi^{-1}(S) = B_M\big(x_0,r\lambda^{-1/d}\big)\), then \(\diam_{M}(C(x;\eta)) \leq r\lambda^{-1/d}/100\).
    \end{enumerate}
\end{proposition}

The rest of this subsection is dedicated to proving Proposition~\ref{prop:comparison}.
We start by coupling the point processes $\xi$ and $\eta$ so that they locally agree with high probability.

\begin{lemma}
\label{lem:pointsequal}
    Assume the conditions of Proposition \ref{prop:comparison} hold.
    Let $\delta, r>0$, then there exists a \(\lambda_0 = \lambda_0(\delta,r)\) such that for all $\lambda \geq \lambda_0$ and all \(x_0 \in M\) there exists a coupling \((\xi, \eta)\) of Poisson point processes
    \[\xi\sim \mathrm{PPP}(\lambda\cdot dx, \R^d) \quad \text{and} \quad \eta \sim \mathrm{PPP}(\lambda\cdot \mu_M, M),\]
    and a chart \(\varphi \in \pazocal{E} (x_0,r\lambda^{-1/d})\) so that
    \begin{equation}
    \P\big( \varphi(\eta) = \xi \cap B_{\R^d}(0, r\lambda^{-1/d})\big)\geq 1-\delta. 
    \end{equation}
\end{lemma}

\begin{proof}
Let \(\lambda_0\) be chosen so that for \(\lambda \geq \lambda_0\) the radius \(r_\lambda \coloneqq r \lambda^{-1/d}\) satisfies \eqref{eq:radius_condition}.
For \(\lambda \geq \lambda_0\) and \(x_0 \in M\) choose a chart \(\varphi \in \pazocal{E}(x_0,r\lambda)\), then Lemma \ref{lem:normal} yields that \(\varphi_*\mu_M\) has a Radon--Nikodym density \(f\) on \(B_{\R^d}(0,r_\lambda)\) that satisfies \( (1-Kr_\lambda^2)^{d/2} \leq f \leq (1+Kr_\lambda^2)^{d/2} \).
Assume that \(\lambda_0\) is chosen large enough to ensure that \(1-c \leq f \leq 1+c\) where \(c=c(r,\delta)\) is a constant that will be specified later.
Let
\begin{align*}
    \tau_0 &\sim \PPP(B_{\R^d}(0,r_\lambda), \lambda (f \wedge 1) \cdot dx), & \tau_3 &\sim \PPP(M \setminus B_{M}(x_0,r_\lambda), \lambda \cdot \mu_M)\\
    \tau_1 &\sim \PPP(B_{\R^d}(0,r_\lambda), \lambda (f - 1)^+ \cdot dx), & \tau_4 &\sim \PPP(\R^d \setminus B_{\R^d}(0,r_\lambda), \lambda \cdot dx),\\
    \tau_2 &\sim \PPP(B_{\R^d}(0,r_\lambda), \lambda (1 - f)^+ \cdot dx),&&
\end{align*}
be independent Poisson processes, where we write \(a \wedge b \coloneqq \min\{a,b\}\) and \(a^+ \coloneqq \max\{a,0\}\) for \(a,b \in \R\).
By the superposition theorem and the mapping theorem for Poisson processes (see, e.g., \cite[Theorem 3.3 and Theorem 5.1]{LastPenrosePP}), it holds that
\[\eta \coloneqq \varphi^{-1}(\tau_0 + \tau_1) + \tau_3 \sim \PPP(M, \lambda \cdot \mu_M),\quad
    \xi \coloneqq \tau_0 + \tau_2 + \tau_4 \sim \PPP(\R^d, \lambda \cdot dx). \]
Now observe that
\begin{align}
    \P\big( \varphi(\eta) = \xi \cap B_{\R^d}(0, r_\lambda)\big)
    &=\P\big(\tau_0 + \tau_1 = \tau_0 + \tau_2 \big)\\
    &\geq \P\big( \tau_1 \equiv 0 \big) \cdot \P\big( \tau_2 \equiv 0 \big)\\
    &=\exp\bigg(-\int_{B_{\R^d}(0, r_\lambda)} \lambda (f-1)^+\, dx -\int_{B_{\R^d}(0, r_\lambda)} \lambda (1-f)^+\, dx\bigg)\\
    &=\exp\bigg(-\int_{B_{\R^d}(0, r_\lambda)} \lambda|f-1|\, dx\bigg)\\
    &\geq\exp\big(-c \lambda |B_{\R^d}(0, r_\lambda)|\big)\\
    &=\exp\big(-c |B_{\R^d}(0, r)|\big)\\
    &= 1-\delta,
\end{align}
for \(c = -\log(1-\delta) / |B_{\R^d}(0, r)|\).
\end{proof}

Next, we show that the points of a Poisson point process in \(\R^d\) are dense in some sense.

\begin{lemma}\label{lem:density}
Let $\delta, C>0$. There exists $r_0 = r_0(\delta,C)>0$ such that for $r\geq r_0$, $\lambda>0$ and $\xi\sim \mathrm{PPP}\big(\lambda dx, \R^d\big)$ we have
\begin{equation}
\P \bigg(\de(x,\xi)\leq r\lambda^{-1/d}/C \; \text{ for all } x\in B_{\R^d}\big(0, r\lambda^{-1/d}\big) \bigg) \geq 1-\delta.
\end{equation}
\end{lemma}
\begin{proof} Fix $\delta,C>0$.
We further assume that \(\lambda=1\), since the claim for general \(\lambda\) follows by rescaling.
Since a ball of radius $r$ has Euclidean volume $\asymp r^d$, we can cover $B_{\R^d}(0,r)$ with roughly $C^d$ balls of radius $r/(2C)$.
More precisely, we can find $x_1,\ldots, x_n\in B_{\R^d}(0,r)$ such that
\[B_{\R^d}(0,r)\subset \bigcup_{i=1}^n B_{\R^d}\big(x_i,r/(2C)\big) \quad \text{and} \quad n\lesssim C^d.\]
(To see this, choose a maximal set of points in \(B_{\R^d}(0,r)\) at distance \(r/(2C)\) and argue as in Lemma \ref{lem:bounded_degree}.)
Next, we note that under the event that each of the balls in the union above contains a point of $\xi$, the event
\[ \big\{d_{\R^d}(x,\xi)\leq r / C \; \text{ for all } x\in B_{\R^d}(0,r)\} \]
occurs.
By the union bound, the complement of this event has probability at most
\[n\cdot \exp\bigg(-|B_{\R^d}(0,r/(2C)\big)|\lesssim C^d\cdot \exp\bigg( -r^d/(2C)^d\bigg),\]
which can be made smaller than $ \delta$ when $r$ is sufficiently large compared to $C$.
\end{proof}

We now introduce the notion of \emph{robustness}, which will be instrumental in the proof of Proposition \ref{prop:comparison}.

\begin{definition}
    Let \(r, \rho > 0\).
    We say that a locally finite point set \(\xi\subset\R^d\) is \emph{\((r,\rho)\)-robust}, when the following holds for all \(x,y \in \xi \cap B_{\R^d}(0,r)\):
    \begin{enumerate}
        \item If the (Euclidean) Voronoi cells of \(x\) and \(y\) intersect in \(B_{\R^d}(0,2r)\), then there exists a \(w_{xy} \in B_{\R^d}(0,2r)\) such that
        \begin{equation*}
            \de(w_{xy},x) = \de(w_{xy}, y) < \de(w_{xy}, \xi\setminus\{x,y\}) - \rho \cdot r.
        \end{equation*}
        \item If the (Euclidean) Voronoi cells of \(x\) and \(y\) do not intersect then
        \begin{equation*}
            \max\{ \de(w,x), \de(w,y) \} > \de(w, \xi) + \rho \cdot r
        \end{equation*}
        holds for all \(w \in B_{\R^d}(0,2r)\).
    \end{enumerate}
\end{definition}

\begin{remark}
    In the previous definition, the point \(w_{xy}\) acts as a \emph{witness} for the fact that \(x\) and \(y\) are neighbors.
    The adjacency relation is further immune to small perturbations (hence the name ``robustness''), since \(w_{xy}\) is guaranteed to be at least a fixed distance \(\rho\cdot r\) closer to \(x\) and \(y\) than to any other point of \(\xi\).
    Similarly, the second part of the definition ensures that non-neighbors fail to be neighbors in a quantifiable way. 
    These two properties will later allow us to replace \(\de\) by \(\varphi_*\dm\) without changing the adjacency relations.

    Bollobás and Riordan use a similar notion of robustness in their proof of \(p_c(\R^2)=1/2\), cf.~\cite[Theorem 6.1]{BR_Voronoi_percolation}.
\end{remark}

\begin{lemma}
\label{lem:robustness}
    Let \(r>0\) and \(\xi \sim \PPP(dx, \R^d)\).
    Then almost surely there is a \(\rho \in (0,1)\) such that \(\xi\) is \((r,\rho)\)-robust.

\end{lemma}
\begin{proof}

    Let \(x,y \in \xi \cap B_{\R^d}(0,r)\) be Voronoi neighbors. 
    Then there is some \(w \in B_{\R^d}(0,2r)\) such that
    \[ \de(w,x) = \de(w,y) \leq \de(w, \xi \setminus \{x,y\}). \]
     Note that the intersection \(C(x;\xi) \cap C(y;\xi)\) is almost surely a \((d-1)\)-dimensional face of \(C(x;\xi)\) and \(C(y;\xi)\) (see \cite[page 6]{moller}) whose relative interior
    \[ \{z \in \R^d \colon \de(z,x) = \de(z, y) < \de(z, \xi\setminus\{x,y\}) \} \]
    contains \(w\) in its closure. 
    Since \(w\in B_{\R^d}(0,2r)\) and \(B_{\R^d}(0,2r)\) is open, it follows that there is some \(w_{xy} \in B_{\R^d}(0,2r)\) and \(\rho^{(1)}_{xy}>0\) with the property that
    \[ \de(w_{xy},x) = \de(w_{xy}, y) < \de(w_{xy}, \xi\setminus\{x,y\}) - \rho^{(1)}_{xy}. \]
    We choose such \(w_{xy}\) and \(\rho^{(1)}_{xy}\) for all pairs of neighbors \(x,y \in B_{\R^d}(0,r)\) whose cells intersect in \(B_{\R^d}(0,2r)\) and let \(\rho^{(1)}\) be the minimum over all those
    (a.s.\ finitely many) pairs.
    It then holds that
    \begin{equation}
    \label{eq:proof_robust_1}
        \de(w_{xy},x) + \de(w_{xy}, y) < \de(w_{xy}, \xi\setminus\{x,y\}) - \rho^{(1)},
    \end{equation}
    for all such neighbors \(x,y\).

    Now assume that the cells of \(x,y \in \xi \cap B_{\R^d}(0,r)\) do not intersect at all.
    Then clearly
    \[ \max\{\de(w,x),\de(w,y)\} > \de(w,\xi) \]
    holds for all \(w\in \R^d\) and hence the continuous function
    \[ w \mapsto \max\{\de(w,x),\de(w,y)\} - \de(w,\xi) \]
    assumes its positive minimum \(\rho^{(2)}_{xy}\) on the compact set \(\overline{B_{\R^d}(0,2r)}\).
    Let \(\rho^{(2)} \coloneqq \frac{1}{2} \min \rho^{(2)}_{xy}\), where the minimum ranges over all pairs of distinct points \(x,y \in \xi \cap B_{\R^d}(0,r)\) that are not
    Voronoi neighbors.
    Then
    \begin{equation}
    \label{eq:proof_robust_2}
        \max\{ \de(w,x), \de(w,y) \} > \de(w, \xi) + \rho^{(2)}
    \end{equation}
    holds for all \(w \in B_{\R^d}(0,2r)\) and all distinct \(x,y \in \xi\cap B_{\R^d}(0,r)\) whose cells do not intersect.
    
    Taking \(\rho \coloneqq \min\{\rho^{(1)}/r,\rho^{(2)}/r\}\), it follows from \eqref{eq:proof_robust_1} and \eqref{eq:proof_robust_2} that \(\xi\) is \((r,\rho)\)-robust.
    Since this implies \((r,\rho')\)-robustness for \(0 < \rho' < \rho\), we can assume w.l.o.g.~that \(\rho \in (0,1)\).
\end{proof}

The following corollary roughly says that, with high probability, the Voronoi diagram is immune to perturbations of order \( \rho \cdot r \lambda^{-1/d}\) (for any intensity \(\lambda > 0\)) for some \(\rho>0\).
Since the difference of the metrics \(|\varphi_* \dm -\de|\) is of order \(O((r \lambda^{-1/d})^2)\) (as \(\lambda \to \infty\)) on \(B_{\R^d}(0,r\lambda^{-1/d})\) by Lemma \ref{lem:normal}, it will be comfortably below this tolerance for large \(\lambda\). 

\begin{corollary}
\label{cor:robustness}
    Let \(r, \delta > 0\).
    Then there exists a $\rho\in (0,1)$ such that for every $\lambda>0$,
    \[ \P\left(\text{\(\xi\) is \((r\lambda^{-1/d},\rho)\)-robust}\right) \geq 1-\delta, \]
    where \(\xi \sim \PPP(\lambda dx, \R^d)\).
\end{corollary}
\begin{proof}
    For \(\lambda = 1\), the statement follows directly from Lemma~\ref{lem:robustness} and the fact that \(\P\) is continuous from below.
    For general \(\lambda > 0\), the claim then follows by a scaling argument:
    If \(\xi \sim \PPP(dx,\R^d)\), then \(\lambda^{-1/d} \xi \sim \PPP(\lambda dx,\R^d)\) and it further holds that \(\xi\) is \((r, \rho)\)-robust if and only if \(\lambda^{-1/d} \xi\) is \((r \lambda^{-1/d}, \rho)\)-robust.
\end{proof}

The following lemma contains some purely deterministic arguments that enter into the proof of Proposition \ref{prop:comparison}.

\begin{lemma}\label{lem:conditions}
    Let \(\xi \subset \R^d\) and \(\eta \subset M\) be locally finite deterministic point sets.
    Let \(\rho \in (0,1)\) and let \(r > 0\) so that \(8r\) satisfies \eqref{eq:radius_condition}.
    Let \(x_0 \in M\) and \(\varphi \in \pazocal{E}(x_0,8r)\).
    Then the following conditions imply that \(\varphi_*\pazocal{G}_{M}(\eta)|_{B_{\R^d}(0,r)} = \pazocal{G}_{\R^d}(\xi)|_{B_{\R^d}(0,r)}\).
    \begin{itemize}
        \item \emph{Equality (EQL):} \(\varphi(\eta) = \xi \cap B_{\R^d}(0,8r)\).
        \item \emph{Robustness (ROB):} \(\xi\) is \((r,2\rho)\)-robust.
        \item \emph{Small cells (SMA):} \(\diam_{\R^d}(C(x,\xi)) \leq r\) for all \(x \in \xi \cap B_{\R^d}(0,r)\) and
                \(\diam_{M}(C(x,\eta)) \leq r\) for all \(x \in \eta \cap B_{M}(x_0,r)\).
        \item \emph{Similarity of metrics (SIM):}   \(\de - \rho \cdot r \leq \varphi_*\d_M \leq \de + \rho \cdot r\) on \(B_{\R^d}(0,8r)\).
    \end{itemize}
\end{lemma}
\begin{proof}
    We use the convention \(\d(x,\emptyset) = \infty\).
    Let \(x,y \in \xi \cap B_{\R^d}(0,r)\) be neighbors in the Euclidean Voronoi diagram.
    Take some \(w\) that lies in the intersection of their Voronoi cells.
    By (SMA) it must hold that \(w \in B_{\R^d}(0,2r)\) and thus (ROB) guarantees the existence of a witness \(w_{xy}\in B_{\R^d}(0,2r)\) so that
    \[ \de(w_{xy},x) = \de(w_{xy},y) < \de(w_{xy},\xi\setminus\{x,y\}) - 2 \rho \cdot r.\]
    Using (EQL), we see that
    \[ \de(w_{xy},\xi\setminus\{x,y\}) \leq \de(w_{xy},\xi\cap B_{\R^d}(0,8r)\setminus\{x,y\}) = \de(w_{xy},\varphi(\eta)\setminus\{x,y\}), \]
    so that
    \[ \de(w_{xy},x) = \de(w_{xy},y) < \de(w_{xy},\varphi(\eta)\setminus\{x,y\}) - 2 \rho \cdot r.\]
    Using (SIM), we now get
    \[ \varphi_*\d_M(w_{xy},x), \varphi_*\d_M(w_{xy},y) <  \varphi_*\d_M(w_{xy},\varphi(\eta)\setminus\{x,y\}). \]
    Let \(\wt x, \wt y, {\wt w}_{xy}\) denote the preimages of \(x,y,w_{xy}\) under \(\varphi\) respectively, then the previous inequality reads
    \[\d_M({\wt w}_{xy},\wt x), \d_M({\wt w}_{xy},\wt y) <  \d_M({\wt w}_{xy}, \eta\cap\dom(\varphi)\setminus\{\wt x,\wt y\}). \]
    Since \(\wt x, \wt y, {\wt w}_{xy}\in B_{M}(x_0,2r)\),
    it follows that \(\d_M({\wt w}_{xy},\wt x), \d_M({\wt w}_{xy},\wt y) < 4r\).
    On the other hand, \(\d_M({\wt w}_{xy},\wt z) \geq 4r\) for any \(\wt z \in M \setminus \dom(\varphi) = M \setminus B_M(x_0,8r)\).
    It thus follows that
    \[ \d_M({\wt w}_{xy},\wt x), \d_M({\wt w}_{xy},\wt y) < \d_M({\wt w}_{xy}, \eta\setminus\{\wt x,\wt y\}). \]
    One then finds a point \({\wt w}'_{xy}\in M\) so that
    \[ \d_M({\wt w}'_{xy},\wt x) = \d_M({\wt w}'_{xy},\wt y) <  \d_M({\wt w}'_{xy},\eta\setminus\{\wt x,\wt y\}).\]
    (Assume w.l.o.g.\ that \(\d_M(\wt w_{xy},\wt x) \leq \d_M(\wt w_{xy}, \wt y)\).
    Take any minimizing geodesic from \(\wt w_{xy}\) to \(\wt y\) and let \(\wt w'_{xy}\) be the point on this curve that is 
    equidistant from \(\wt x\) and \(\wt y\).
    It follows from the triangle inequality that this point is closer to \(\wt x\) and \(\wt y\) than to any other point of \(\eta\).)
    Hence, \(\wt x\) and \(\wt y\) are neighbors in \(\pazocal G_{M}(\eta)\).
    
    Conversely, assume that there exist \(\wt x, \wt y \in \eta \cap B_{M}(x_0,r)\) which are neighbors in \(\pazocal G_{M}(\eta)\)
    so that \(x = \varphi(\wt x)\) and \(y = \varphi(\wt y)\) are not neighbors in \(\pazocal G_{\R^d}(\xi)\).
    By (SMA), there must be a witness \(\wt w_{xy} \in B_{M}(x_0,2r)\) so that
    \[ \d_M(\wt w_{xy},\wt x) = \d_M(\wt w_{xy},\wt y) \leq \d_M(\wt w_{xy}, \eta) \leq \d_M(\wt w_{xy},\eta\cap B_M(x_0,8r)). \]
    Writing \(w_{xy} = \varphi(\wt w_{xy})\) and using (EQL), it follows that
    \[ \varphi_*\d_M(w_{xy}, x) = \varphi_*\d_M(w_{xy},y) \leq \varphi_*\d_M(w_{xy}, \xi \cap B_{\R^d}(0,8r)). \]
    Applying (SIM) now yields
    \[ \de(w_{xy},x), \de(w_{xy},y) \leq \de(w_{xy},\xi\cap B_{\R^d}(0,8r)) + 2 \rho \cdot r. \]
    Since \(w_{xy},x,y \in B_{\R^d}(0,2r)\), we can argue the same way we did earlier to obtain
    \[ \de(w_{xy},x), \de(w_{xy},y) \leq \de(w_{xy},\xi) + 2 \rho \cdot r, \]
    which directly contradicts (ROB).
\end{proof}

\begin{proof}[Proof of Proposition \ref{prop:comparison}]
Fix $\delta>0$.
First, using Lemma~\ref{lem:density} (with $C=8000$) choose $r_0$ such that for all $r\geq r_0$ and $\lambda>0$ we have
\begin{equation}\label{eq:dens}
\P \bigg(\de(x,\xi)\leq r\lambda^{-1/d}/1000 \; \text{ for all } x\in B_{\R^d}\left(0, 8r\lambda^{-1/d}\right) \bigg) \geq 1-\delta/3
\end{equation}
for \(\xi \sim \PPP(\lambda dx, \R^d)\).
Now fix $r\geq r_0$.
Corollary \ref{cor:robustness} yields a \(\rho \in (0,1)\) so that for all \(\lambda > 0\) it holds that
\begin{equation}\label{eq:robustness}
  \P\left(\xi \text{ is $( r\lambda^{-1/d},2\rho)$-robust} \right)\geq 1-\delta/3
\end{equation}
for  \(\xi \sim \PPP(\lambda dx, \R^d)\).
Using Lemma \ref{lem:pointsequal}, we obtain a \(\lambda_0 > 0\)
such that for all $\lambda \geq \lambda_0$ and all \(x_0 \in M\) there exists a coupling \((\xi, \eta)\) of Poisson point processes
    \[\xi\sim \mathrm{PPP}(\lambda\cdot dx, \R^d) \quad \text{and} \quad \eta \sim \mathrm{PPP}(\lambda\cdot \mu_M, M),\]
and a chart \(\wt \varphi \in \pazocal{E} (x_0,8r\lambda^{-1/d})\) so that
\begin{equation}\label{eq:equality}
\P\big( \wt\varphi(\eta) = \xi \cap B_{\R^d}(0, 8r\lambda^{-1/d})\big)\geq 1-\delta/3. 
\end{equation}
Without loss of generality, we may assume that \(\lambda_0\) is large enough so that Lemma \ref{lem:normal} guarantees
\begin{equation}\label{eq:length_similar}
    |\wt\varphi_*\d_M - \de| \leq r \lambda^{-1/d} (\rho \wedge 10^{-3}) \quad \text{on} \quad B_{\R^d}(0,8r\lambda^{-1/d})
\end{equation}
for \(\lambda\geq \lambda_0\) and \(\wt\varphi\) as above. 

Now fix \(\lambda \geq \lambda_0\) and take \(\eta, \xi\) and \(\wt\varphi\) as above.
We show that under the intersection of the events in \eqref{eq:dens}, \eqref{eq:robustness} and \eqref{eq:equality} the conclusions in Proposition \ref{prop:comparison} hold for \(
\varphi \coloneqq \wt\varphi|_{B_M(x_0,r\lambda^{-1/d})}\).
Since by the union bound this intersection has probability at least \(1-\delta\), this concludes the proof.

If those three events occur, then it follows using \eqref{eq:length_similar} that \(\xi\) and \(\eta\) are sufficiently dense (w.r.t.~\(\de\) and \(\dm\) respectively) in \(B_{\R^d}(0,8r\lambda^{-1/d})\) (respectively \(B_M(x_0,8r\lambda^{-1/d})\)) as to guarantee the second and third conclusions in Proposition~\ref{prop:comparison}, as well as the condition (SMA) in Lemma~\ref{lem:conditions}.
The remaining conditions (EQL), (ROB) and (SIM) are satisfied as well and thus Lemma~\ref{lem:conditions} (applied with \(\wt\varphi\)) implies the first conclusion in Proposition~\ref{prop:comparison}.

\end{proof}

\section{Fine-graining}
\label{sec:fine_graining}

In this section, we prove Theorem~\ref{thm:main}. The proof consists of two fine-graining arguments, a subcritical and a supercritical one, on the embedded graphs $G_\ep$ defined in Section~\ref{sec:discretization}. These arguments are presented in separate subsections.

We first introduce some notation. We write $\Omega_{M}$ for the set of all locally finite subsets of $M$.
We view a pair $\overline{\omega}=(\omega_w,\omega_b)\in \Omega^2_{M}$ as a percolation configuration where the points in $\omega_w$ are white and the points in $\omega_b$ are black. Given $\overline{\omega}\in \Omega^2_{M}$, we denote by $\pazocal{V}_{M}(\overline{\omega})\subset M$ the white region, that is, the union of the cells of $\omega_w$ in the Voronoi tessellation corresponding to $\omega\coloneqq \omega_w\cup\omega_b$. In particular, it holds for $\overline{\eta}(\lambda,p) \coloneqq (\eta_w(\lambda,p),\eta_b(\lambda,p))$ where $\eta_w(\lambda,p)$ and $\eta_b(\lambda,p)$ are Poisson point processes on $M$ with intensities $p\lambda\mu_M$ and $(1-p)\lambda\mu_M$, respectively, that $\pazocal{V}_{M}(\overline{\eta}(\lambda,p))$ is distributed as $\pazocal{V}_{M}(\lambda,p)$, as defined in the introduction.

\subsection{Subcritical fine-graining}
\label{sec:subcritical}

In this section, we prove the subcritical part of Theorem~\ref{thm:main}, which is encapsulated in the following proposition. In order to prove that a certain ``good local event'' occurs with high probability, we combine the \emph{subcritical} sharpness for Voronoi percolation on $\R^d$ proved in \cite{DRT17} together with the local comparison established in Proposition~\ref{prop:comparison}.

\begin{proposition}\label{prop:subcritical}
   $\liminf_{\lambda\rightarrow \infty} p_c(M,\lambda)\geq p_c(\R^d).$
\end{proposition}

\begin{proof}[Proof of Proposition~\ref{prop:subcritical}]
   Given $\ep>0$ and $x\in G_\varepsilon$, consider the event 
   \begin{equation}
       F^{sub}_{\ep}(x) \coloneqq \{\overline{\omega}\in \Omega^2_{M}:~ \pazocal{V}_{M}(\overline{\omega}')\notin \mathrm{Cross}(x,\varepsilon) ~~\forall\, \overline{\omega}'\in\Omega^2_{M} \text{ such that } \overline{\omega}_{|B_{M}(x,10\varepsilon)}=\overline{\omega}'_{|B_{M}(x,10\varepsilon)}\},
   \end{equation}
   where 
   $$\mathrm{Cross}(x,\varepsilon)\coloneqq \{ V\subset M:~\lr{}{V}{B_{M}(x,2\ep)}{\partial B_{M}(x,5\ep)} \}.$$ In words, this is the event that there is no white crossing of the annulus $B_M(x,5\ep)\setminus B_M(x,2\ep)$, and that this is determined by points in $B_M(x, 10\ep)$.
   The main ingredient in this proof is the following Lemma. 
   \begin{lemma}\label{lem:seed=good_sub}
       For all $p<p_c(\R^d)$ and $\delta>0$, there exists $l=l(p,\delta)<\infty$ and $\lambda_0=\lambda_0(p,\delta)>0$ such that we have 
       $$\forall x_0\in M \text{ and } \lambda\geq \lambda_0 \quad \P\big(\overline{\eta}(\lambda,p)\in F^{sub}_{\ep}(x_0)\big)\geq 1-\delta,$$ where $\varepsilon= l\lambda^{-1/d}$. 
   \end{lemma}
   Before proving Lemma~\ref{lem:seed=good_sub}, we finish the proof of Proposition~\ref{prop:subcritical}.  Fix $p<p_c(\R^d)$. Recall from Lemmas~\ref{lem:bounded_degree} and \ref{lem: bound on Delta} that the degrees in $G_\varepsilon$ are bounded uniformly by $\Delta \coloneqq \sup_{\ep\in(0,1]}\Delta(\ep)<\infty$ for every $\varepsilon\in(0,1]$. Therefore, by the main result of \cite{LigSchSta97}\footnote{The use of \cite{LigSchSta97} here, as well as later in the proof of the supercritical fine-graining, can be avoided. We choose to use it in order to separate the proofs based on counting arguments (Section~\ref{sec:discretization}) from those relying on renormalization techniques (Section~\ref{sec:fine_graining}). See however Appendix~\ref{app:phase-transition} for a coarse-graining argument that does not use \cite{LigSchSta97}.}, there exists a $\delta_1>0$ such that every $10\Delta$-dependent site-percolation model $Z$ on $G_\varepsilon$ satisfying $\P(Z(x)=1)\leq \delta_1$ for all $x\in G_\varepsilon$ is stochastically dominated by a Bernoulli percolation of parameter $\delta_0/2$, where $\delta_0$ is given by Theorem~\ref{thm:discretization_percolation}.
   
   Now, fix $l$ and $\lambda_{0}$ that we get by applying Lemma~\ref{lem:seed=good_sub} with $\delta=\delta_1$. For every $\lambda\geq\lambda_0$, consider the site percolation model $Z_\lambda$ on $G_\varepsilon$ (where $\ep=l\lambda^{-1/d}$) given by 
   \begin{equation}
       Z_\lambda(x) \coloneqq \mathbbm{1}_{F^{sub}_{\ep}(x)^c}(\overline{\eta}(\lambda,p)),
   \end{equation}
   which satisfies $\P(\omega_\lambda(x)=1)\leq \delta_1$ by Lemma~\ref{lem:seed=good_sub}. 
   By Lemma~\ref{lem:quasi_isometry}, $\d_{G_\ep}(x,y)\geq 10\Delta$ implies $\dm(x,y)\geq 20\varepsilon$, so $Z_\lambda$ is $10\Delta$-dependent. Therefore, for every $\lambda\geq\lambda_0$, $Z_\lambda$ is stochastically dominated by a Bernoulli site percolation with parameter $\delta_0/2$ and so does not percolate by Theorem~\ref{thm:discretization_percolation}. 
   However, if $\pazocal{V}_{M}(\overline{\eta}(\lambda,p))$ has an infinite white path, then we can use Lemma~\ref{lem:connected_sets} to obtain an infinite open path in $Z_\lambda$. Therefore, $p_c(M,\lambda)\geq p$ for every $\lambda\geq\lambda_0$. Since $p<p_c(\R^d)$ is arbitrary, the result follows.
\end{proof}

\begin{proof}[Proof of Lemma~\ref{lem:seed=good_sub}]
    Let $p<p_c$ and $\delta>0$. Let $\pazocal{V}_{\R^d}(\lambda,p)$ denote the union of white cells in Voronoi percolation on $\R^d$ of parameter $p$ and intensity $\lambda$.
    It follows from \cite[Theorem 1]{DRT17} that
    \begin{equation}\label{eq:sub_sharpness}
        \P(\lr{}{\pazocal{V}_{\R^d}(1,p)}{B_{\R^d}(0,3R)}{\partial B_{\R^d}(0,4R)}) \xrightarrow{R\to\infty} 0.
    \end{equation} 
    More precisely, \cite[Theorem 1]{DRT17} states that there is a constant \(c = c_p>0\) so that $\P(\lr{}{\pazocal{V}_{\R^d}(1,p)}{0}{\partial B_{\R^d}(0,R)})\leq e^{-cR}$, \(R \geq 1\), from which it is easy to deduce \eqref{eq:sub_sharpness}.
	Indeed, by the FKG inequality and translation invariance, for every $x\in\R^d$ and $R\geq 1$, one has $\P(\lr{}{\pazocal{V}_{\R^d}(1,p)}{B_{\R^d}(x,1)}{\partial B_{\R^d}(x,R)})\leq e^{-cR}/\P(B_{\R^d}(x,1) \text{ fully white})= Ce^{-cR}$.
	Taking a union bound over $x\in \frac{1}{\sqrt d}\Z^d\cap B_{\R^d}(0,3R)$ then yields \eqref{eq:sub_sharpness}. 
    By the rescaling property of the Poisson--Voronoi tessellation on $\R^d$, it follows from \eqref{eq:sub_sharpness} that there exists $R_0=R_0(p,\delta)<\infty$ such that for all \(R \geq R_0\),
    \begin{equation}\label{eq:sub}
        \P(\lr{}{\pazocal{V}_{\R^d}(\lambda,p)}{B_{\R^d}(0,3R\lambda^{-1/d})}{\partial B_{\R^d}(0,4R\lambda^{-1/d})})\leq \delta ~~~ \text{ for all } \lambda>0.
    \end{equation}

    Let $r_0=r_0(\delta)$ be the radius in Proposition~\ref{prop:comparison}. We show that for $l=\max(R_0,r_0)$, the conclusion in Lemma \ref{lem:seed=good_sub} holds.  For $r=20l$, Proposition~\ref{prop:comparison} gives us a coupling $(\xi,\eta)$ with marginals given by \eqref{eq:coupling_marginals} (and a normal coordinate chart $\varphi \in \pazocal{E}(x_0,r\lambda^{-1/d})$) such that for $\lambda$ large enough, the coupling satisfies properties \ref{property 1 coupling}, \ref{property 2 coupling} and \ref{property 3 coupling} of Proposition~\ref{prop:comparison} with probability at least $1-\delta$. Let $E$ be the event that these properties hold, so we have $\P(E)\geq 1-\delta$ for all $\lambda\geq \lambda_0(r,\delta)$.

    Let \(S \coloneqq B_{\R^d}(0,r\lambda^{-1/d}) = \im(\varphi)\).
   Now, on the event $E$, we color the points in $\xi$ and $\eta$ as follows. We color each point in $\xi|_S$ white or black independently with probability $p$ and $1-p$, respectively. This induces a coloring of $\eta|_{\varphi^{-1}(S)}$ by assigning a point $x$ the color of $\varphi(x)$.\footnote{For the sake of defining a global coupling of colored points, on the event $E$ we color the points in $\xi|_{S^c}$ and $\eta|_{\varphi^{-1}(S)^c}$ independently, and on the event $E^c$ we color all points in $\eta$ and $\xi$ independently.} We thus obtain coupled configurations $(\overline{\eta},\overline{\xi})$. In what follows, write $\ep=l\lambda^{-1/d}$. Define
    $$A\coloneqq \{\nlr{}{~\pazocal{V}_{\R^d}(\overline{\xi})~}{B_{\R^d}(0,3\varepsilon)}{\partial B_{\R^d}(0,4\varepsilon)}\}.$$
    Since $l\geq R_0$, \eqref{eq:sub} gives that $\P(A)\geq 1-\delta$.
    We claim that
    \begin{equation}\label{eq:lemma_sub_inclusion}
        A\cap E \subset \{\overline{\eta}\in F^{sub}_{\ep}(x_0)\}.
    \end{equation}
    Taking this claim for granted gives that $\P(\overline{\eta}\in F^{sub}_{\ep}(x_0))\geq 1-2\delta$ for $\lambda \geq \lambda_0(r,\delta)$, which proves the lemma since $\delta>0$ is arbitrary. 
    
  We now prove the inclusion \eqref{eq:lemma_sub_inclusion}. We claim that on the event $E$, if there exists a white crossing of $B_M(x_0,5\ep)\setminus B_{M}(x_0,2\ep)$  in $\pazocal{V}_{M}(\overline{\eta})$ then there is a white crossing of $B_{\R^d}(0, 4\varepsilon)\setminus B_{\R^d}(0,3\ep)$ in $\pazocal{V}_{\R^d}(\overline{\xi})$.

   To prove this, let $\gamma$ be a white path in $\pazocal{V}_{M} (\overline{\eta})$ crossing $B_{M}(x_0,5\ep)\setminus B_{M}(x_0,2\ep)$. Using property (3) of the coupling and arguing as in Lemma \ref{lem:connected_sets}, $\gamma$ induces a white path $\pi$ in $\pazocal{G}_M(\eta)$ such that the first vertex of $\pi$ is in $B_{M}(x_0,3\ep)$ and the last is in $M\setminus B_M(x_0,4\ep)$. Property (1) and the definition of the coloring implies that $\varphi(\pi)$ is a white path in $\pazocal{G}_{\R^d}(\xi)$. Now, since $\varphi$ is a normal coordinate chart, $\varphi_* \mathrm{d}_M(0,x)=\mathrm{d}_{\R^d}(0,x)$ for all $x\in B_{\R^d}(0,r)$. Hence, $\varphi(\pi)$ is a white path in \(\pazocal G_{\R^d }(\xi)\) whose first vertex is in $B_{\R^d}(0,3\ep)$ and last vertex is in $\R^d\setminus B_{\R^d}(0,4\ep)$. Finally, by the definition of $\pazocal{G}_{\R^d}(\xi)$, we can use $\varphi(\pi)$ to obtain a continuous path \(\gamma'\) in \(\pazocal{V}_{\R^d}(\overline{\xi})\) that crosses $B_{\R^d}(0, 4\varepsilon)\setminus B_{\R^d}(0,3\ep)$.
    
    Furthermore, the event $E$ guarantees that the cells intersecting $B_{M}(x_0,5\varepsilon)$ have diameter smaller than $\varepsilon/5$, so they are completely determined by the points contained in $B_{M}(x_0,10\varepsilon)$. This shows \eqref{eq:lemma_sub_inclusion}, concluding the proof.
\end{proof}

\begin{remark}\label{rem:sub_fine-graining_minimalassumption}
We note that the proof of Proposition \ref{prop:subcritical} only required that $p_c(G_\ep)\geq \delta_0$. By Remark~\ref{rem:pc>0_minimalassumption}, the assumptions \ref{thm cond : connected+one-ended} and \ref{thm cond: log expansion} of Theorem \ref{thm:main} could be dropped. 
\end{remark}

\subsection{Supercritical fine-graining}
\label{sec:supercritical}

In this section, we complete the proof of Theorem~\ref{thm:main}. Similarly to the subcritical part, here the occurrence of the ``good local event'' follows from the recently proven \cite{DS25} \emph{supercritical} sharpness for Voronoi percolation on $\R^d$, combined with the local comparison given by Proposition~\ref{prop:comparison}. 

\begin{proposition}\label{prop:supercritical}
$\limsup_{\lambda\rightarrow \infty} p_u(M,\lambda)\leq p_c(\R^d)$.
\end{proposition}

We use the same notation as in the previous section. 

\begin{proof}[Proof of Proposition~\ref{prop:supercritical}]
Given $\ep>0$ and $x\in G_\varepsilon$, consider the event \begin{equation}
       F^{sup}_{\ep}(x) \coloneqq \{\overline{\omega}\in \Omega^2_{M}:~ \pazocal{V}_{M}(\overline{\omega}')\in \mathrm{U}(x,\varepsilon) ~~\forall\, \overline{\omega}'\in\Omega^2_{M} \text{ such that } \overline{\omega}_{|B_{M}(x,20\varepsilon)}=\overline{\omega}'_{|B_{M}(x,20\varepsilon)}\},
   \end{equation}
   where \begin{align*}
\mathrm{U}(x,\varepsilon) := & \left\{ V \subset M : \; \text{there exists a component of } V \text{ crossing the annulus } \right. \\
& B_{M}(x, 16\varepsilon) \setminus B_{M}(x, 2\varepsilon), \; \text{which is the only component of } V \cap B_{M}(x, 18\varepsilon) \\
& \left. \text{crossing the annulus } B_{M}(x, 9\varepsilon) \setminus B_{M}(x, 6\varepsilon) \right\}.
\end{align*}
We will prove that the following holds.
   \begin{lemma}\label{lem:seed=good_sup}
       For every $p>p_c(\R^d)$ and $\delta>0$, there exists $l=l(p,\delta)<\infty$ and $\lambda_0=\lambda_0(p,\delta)>0$ such that we have 
       $$\forall x_0\in M \text{ and } \lambda\geq \lambda_0 \quad \P(\overline{\eta}(\lambda,p)\in F^{sup}_{\ep}(x_0))\geq 1-\delta,$$ where $\ep= l\lambda^{-1/d}$.
   \end{lemma}
   
   Before proving Lemma~\ref{lem:seed=good_sup}, we finish the proof of Proposition~\ref{prop:supercritical}.  Fix $p>p_c(\R^d)$. Recall from Lemma~\ref{lem:bounded_degree} that the degrees in $G_\varepsilon$ are bounded uniformly by $\Delta = \sup_{\ep \in (0,1]}\Delta(\ep)$ for every $\varepsilon\in(0,1]$. Therefore, by the main result of \cite{LigSchSta97}, there exists a $\delta_1>0$ such that every $20\Delta$-dependent site-percolation model $Z$ on $G_\varepsilon$ satisfying $\P(Z(x)=1)\geq 1-\delta_1$ for all $x\in G_\varepsilon$ stochastically dominates a Bernoulli percolation of parameter $1-\delta_0/2$, where $\delta_0$ is given by Theorem~\ref{thm:discretization_percolation}.
   
   Now, let $l$ and $\lambda_{0}$ be given by  Lemma~\ref{lem:seed=good_sup} with $\delta=\delta_1$. For every $\lambda\geq\lambda_0$, consider the site percolation model $Z_\lambda$ on $G_\varepsilon$ (where $\ep=l\lambda^{-1/d}$) given by 
   \begin{equation}
       Z_\lambda(x) \coloneqq \mathbbm{1}_{F^{sup}_{\ep}(x)}(\overline{\eta}(\lambda,p)).
   \end{equation}
   As in the proof of Proposition \ref{prop:subcritical}, it follows using Lemma~\ref{lem:quasi_isometry} that $Z_\lambda$ is $20\Delta$-dependent. Therefore, for every $\lambda\geq\lambda_0$, $Z_\lambda$ stochastically dominates a Bernoulli percolation with parameter $1-\delta_0/2$. Now, it follows from the definition of $\overline{p}_u(G_\ep)$ that this site percolation has a unique infinite open component whose complement has no infinite component. Finally, we show that this implies that $\pazocal{V}_M(\overline{\eta})$ has a unique unbounded component, completing the proof. First, to see that $\pazocal{V}_M(\overline{\eta})$ has an infinite component, let $\pi$ be an infinite path in $G_\ep$ all of whose vertices $x$ satisfy $Z_{\lambda}(x)=1$. For each $x\in \pi$, let $V_x$ be the component in the definition of $U(x,\ep)$. Let $x$ and $y$ be neighboring vertices in $\pi$, which means $\d_M(x,y)\leq 4\ep$, by the definition of $G_\ep$. It follows that $V_x$ crosses $B_M(y, 9\ep) \setminus B_M(y,6\ep)$ and, by the uniqueness part of the definition of $U(y,\ep)$, $V_x$ is connected to $V_y$ in $\pazocal V_M(\overline{\eta})$. So $\cup_{x\in \pi} V_x$ is an unbounded connected subset of $\pazocal V_M(\overline{\eta})$, as required. Second, we show that there is a unique unbounded component in $\pazocal V_M(\overline{\eta})$. Let $\pazocal{C}$ be the unbounded component in $\pazocal V_M(\overline{\eta})$ containing $\cup_{x\in \pi} V_x$.
   Suppose that there is a different unbounded component $\pazocal{C}'$.
   Let $\gamma'$ be an infinite path in $\pazocal{C}'$.
   By Lemma \ref{lem:connected_sets}, we may associate to it a discrete path $\pi'$ in $G_\ep$ such that for all $x\in \pi'$, $\d_M(\gamma', x)\leq 2\ep$.
   Since the complement of the unique open cluster in the site percolation \(Z_\lambda\) contains no infinite connected sets, it follows that \(\pi'\) intersects the infinite open cluster.
   Now arguing as in the existence part above, we obtain that $\gamma'$ is connected to $\pazocal{C}$ in $\pazocal{V}_M(\overline{\eta})$, which contradicts the assumption that $\pazocal{C}$ and $\pazocal{C}'$ are different components of $\pazocal{V}_M(\overline{\eta})$.
   This completes the proof of Proposition~\ref{prop:supercritical}.
\end{proof}

\begin{proof}[Proof of Lemma~\ref{lem:seed=good_sup}]
Fix $p>p_c(\R^d)$ and $\delta>0$.
It follows easily from \cite{BR_Voronoi_percolation} (for $d=2$) and \cite{DS25} (for $d\geq3$) that for every $p>p_c(\R^d)$, 
	\begin{equation}\label{eq:sup_sharpness}
        \P(\pazocal{V}_{\R^d}(1,p)\in \tilde{\mathrm{U}}(R)) \xrightarrow{R\to\infty} 1,
    \end{equation}
    where $\tilde{\mathrm{U}}(R):= \{V\subset\R^d:~\text{there exists a component of } V \text{ crossing the annulus } B_{\R^d}(0,18R)\setminus B_{\R^d}(0,R), \text{ which is the only component of } V\cap B_{\R^d}(0,9R) \text{ crossing the annulus } B_{\R^d}(0,8R)\setminus B_{\R^d}(0,7R) \}$.
    By rescaling, it follows from \eqref{eq:sup_sharpness} that there exists $R_0=R_0(p,\delta)<\infty$ such that for all \(R \geq R_0\),
    \begin{equation}\label{eq:sup}
       \P\big(\pazocal{V}_{\R^d}(\lambda,p)\in \tilde{\mathrm{U}}(R\lambda^{-1/d})\big)\geq 1- \delta ~~~ \text{ for all } \lambda>0.
    \end{equation}

    We quickly explain how to deduce \eqref{eq:sup_sharpness} from \cite{BR_Voronoi_percolation} and \cite{DS25}. For $d=2$, \cite[Theorem 1.2]{BR_Voronoi_percolation} together with a standard union bound implies that for every $p>1/2=p_c(\R^2)$, large macroscopic annuli contain a surrounding  circuit with (exponentially) high probability. By 2D topology, the existence of a circuit implies the uniqueness of the crossing cluster, and \eqref{eq:sup_sharpness} follows.
    As for $d\geq3$, \cite[Theorem 1.2]{DS25} states that $  \P(\pazocal{V}_{\R^d}(1,p)\in \pazocal{U}(L)) \xrightarrow{L\to\infty} 1$ for a very similar event $\pazocal{U}(L)$, which differs from $\tilde{U}(R)$ in the norm and aspect ratio of the annuli involved, as well as the fact that $\tilde{U}(R)$ includes the existence of a crossing in a larger annulus. It is easy to include the existence of such a crossing in $\pazocal{U}(L)$ since this trivially holds with high probability in the supercritical phase. It is then straightforward (arguing as in the proof of Proposition \ref{prop:supercritical}) to realize $\tilde{U}(R)$ by intersecting a bounded number of translates of $\pazocal{U}(L)$ (with, say, $L=R/100\sqrt{d}$), so that \eqref{eq:sup_sharpness} follows easily by a union bound.

 Let $r_0=r_0(\delta)$ be the radius in Proposition~\ref{prop:comparison}. We show that for $l=\max(R_0,r_0)$, the conclusion in Lemma~\ref{lem:seed=good_sup} holds. For $r=20l$, Proposition~\ref{prop:comparison} gives us a coupling $(\xi,\eta)$ with marginals given by \eqref{eq:coupling_marginals} (and normal coordinate chart $\varphi \in \pazocal{E}(x_0,r\lambda^{-1/d})$) such that for $\lambda$ large enough, the coupling satisfies properties \ref{property 1 coupling}, \ref{property 2 coupling} and \ref{property 3 coupling} of Proposition~\ref{prop:comparison} with probability at least $1-\delta$. Let $E$ be the event that these properties hold, so we have $\P(E)\geq 1-\delta$ for all $\lambda\geq \lambda_0(r,\delta)$.
We can also construct a corresponding coupling of colored configurations $(\overline{\eta},\overline{\xi})$ in the same way as in the proof of Lemma~\ref{lem:seed=good_sub}. In what follows, we write $\ep=l\lambda^{-1/d}$. Define
    $$A\coloneqq \{\pazocal{V}_{\R^d}(\overline{\xi})\in \tilde{\mathrm{U}}(\varepsilon)\}.$$
    Since $l\geq R_0$, \eqref{eq:sup} shows that $\P(A)\geq 1-\delta$. We claim that
    \begin{equation}\label{eq:lemma_sup_inclusion}
        A\cap E \subset \{\overline{\eta}\in F^{sup}_{\ep}(x_0)\}.
    \end{equation}
    This inclusion then gives $\P(\overline{\eta}\in F^{sup}_{\ep}(x_0))\geq 1-2\delta$ for $\lambda \geq\lambda_0(r,\delta)$, which completes the proof since $\delta>0$ is arbitrary. 

We now prove this inclusion. Suppose $A\cap E$ holds.  By the definition of $A$, there exists a crossing of $B_{\R^d}(0,18\ep)\setminus B_{\R^d}(0,\ep)$ in $\pazocal{V}_{\R^d}(\overline{\xi})$. Arguing as in the proof of Lemma~\ref{lem:seed=good_sub} we get a crossing of $B_{M}(x_0,16\ep)\setminus B_{M}(x_0,2\ep)$ in $\pazocal{V}_{M}(\overline{\eta})$. This verifies the first requirement in the definition of $F^{sup}_\ep(x_0)$.

To verify the second requirement, suppose that there exist two different components $\pazocal{V}_{M} (\overline{\eta})$ that cross $B_M(x_0,9\ep)\setminus B_M(x_0, 6\ep)$. Arguing as in the proof of Lemma~\ref{lem:seed=good_sub}, we obtain two different components in $\pazocal{V}_{\R^d}(\overline{\xi}) \cap B_{\R^d}(0,19\ep)$ that cross $B_{\R^d}(0,8\ep)\setminus B_{\R^d}(0,7\ep)$. This contradicts the assumption that $A\cap E$ holds. 

Finally, the event $E$ guarantees that the cells intersecting $B_{M}(x_0,18\varepsilon)$ have diameter smaller than $\varepsilon/5$, so they are completely determined by the points of $\overline{\eta}$ contained in ${B_{M}(0,20\varepsilon)}$. This completes the proof.
\end{proof}

\begin{remark}\label{rem:sup_fine-graining_pc_only}
If we only wanted to prove that $\limsup_{\lambda\to\infty}p_c(M,\lambda)\leq p_c(\R^d)$, it would be enough to have $p_c(G_\ep)\leq 1-\delta_0$ for some family of embedded graphs of mesh-size $\ep \in (0,1]$. For example, we can construct such a family assuming that $M$ contains an isometric copy of some manifold satisfying the assumptions of Theorem~\ref{thm:main}. See also Remark~\ref{rem:exp_cutsets_transience} for another class of manifolds for which one could hope to prove such a statement.
\end{remark}

\appendix

\section{Normal coordinates}\label{app:normal coordinates}
For normal coordinates, the following comparison theorem is available:

\begin{theorem}
\label{thm:metric_comparison}
    Let \(x_0 \in M\), \(r_0 \leq r_\text{inj}(M)\) and \(\varphi \in \pazocal{E}(x_0,r_0)\).
    Let \(s_c\) be the function defined in \eqref{eq:def_sc}.
    For \(g\) expressed in \(\varphi\)-coordinates, the following bounds hold:
    \begin{enumerate}
        \item If all sectional curvatures are bounded above by a constant \(c = 1/R^2 > 0\) and \(r_0 \leq \pi R\), then for all \(x \in B_M(x_0,r_0) \setminus \{x_0\}\) and
        \(w \in T_x M\), it holds that \(g(w,w) \geq (s_{c}(r_0)^2/r_0^2)\langle w,w \rangle\).
        \item If all sectional curvatures are bounded below by a constant \(c \leq 0\), then for all \(x \in B_M(x_0,r_0) \setminus \{x_0\}\) and
        \(w \in T_x M\), it holds that \(g(w,w) \leq (s_{c}(r_0)^2/r_0^2)\langle w,w \rangle\).
    \end{enumerate}
\end{theorem}
\begin{proof}
    From \cite[Theorem 11.10]{LeeRM} it follows, under the conditions of the first claim, that \(g(w,w) \geq g_{c}(w,w)\) where \(g_c\)
    denotes the constant-curvature \(c\) metric in \(\varphi\)-coordinates.
    Now, using \cite[Theorem 10.14]{LeeRM} (and the notation introduced in the paragraphs preceding this theorem), we can bound \(g_c\) by
    \[ g_c = d\, r^2 + s_c(r)^2 \hat{g} \geq \frac{s_c(r)^2}{r^2} (d\, r^2 + r^2 \hat{g}) =  \frac{s_c(r)^2}{r^2} g_0, \]
    where we used the fact that \(s_c(r) \leq r\) for \(c > 0\) and write \(g' \leq g''\) when \(g''-g'\) is positive semi-definite. 
    Since further \({s_c(r)^2}/{r^2} \geq {s_c(r_0)^2}/{r_0^2}\) for \(r \leq r_0\), the first part of the claim follows.

    For the second one, we argue analogously.
    This time, we bound
    \[ g \leq g_c = d\, r^2 + s_c(r)^2 \hat{g} \leq \frac{s_c(r)^2}{r^2} (d\, r^2 + r^2 \hat{g}) = \frac{s_c(r)^2}{r^2} g_0, \]
    where we used that \(s_c(r) \geq r\) for \(c \leq 0\).
    Since \({s_c(r)^2}/{r^2} \leq {s_c(r_0)^2}/{r_0^2}\) for \(r \leq r_0\), the second part follows.
\end{proof}

We can now give the proof of Lemma \ref{lem:normal}:

\begin{proof}[Proof of Lemma \ref{lem:normal}]
    The first claim follows from \cite[Corollary 6.13]{LeeRM}.
    From \cite[Theorem IX.6.1.]{Chavel} and \eqref{eq:radius_condition}, it follows that \(B_M(x_0,r)\) is convex, i.e., the minimizing geodesic between any two points \(x,y \in B_M(x_0,r)\) is unique and contained in \(B_M(x_0,r)\).
    Obviously, \(B_{\R^d}(0,r)\) is convex as well.
    It then follows using Theorem \ref{thm:metric_comparison} that
    \[ \frac{s_{K}(r)}{r} \de(x,y) \leq \varphi_* \d_M(x,y) \leq \frac{s_{(-K)}(r)}{r} \de(x,y), \]
    for all \(x,y \in B_{\R^d}(0,r)\).
    A Taylor expansion of the sine function gives \(\sin(t) \geq t - \frac{1}{6}t^3\), \(t\geq 0\), from which we obtain
    \[ \frac{s_{K}(r)}{r} = \frac{\sin(\sqrt{K}r)}{\sqrt{K}r} \geq 1 - \frac{1}{6} Kr^2 \geq 1-Kr^2, \]
    for \(r > 0\).
    Similarly, on sees that \(\sinh(t) \leq t + \frac{1}{6}\cosh(t)t^3\), \(t\geq 0\), which yields
    \[ \frac{s_{(-K)}(r)}{r} = \frac{\sinh(\sqrt{K}r)}{\sqrt{K}r} \leq 1 + \frac{1}{6}\cosh(\sqrt{K}r) Kr^2 \leq 1+Kr^2, \]
    for \(0<r \leq {\pi}/({2\sqrt{K}})\).
    The second claim follows immediately.

    For the third claim, recall that the volume can be expressed in \(\varphi\)-coordinates as
    \[ \varphi_* \mu_M(A) = \int_A \sqrt{\det(g)} dx. \]
    By Theorem \ref{thm:metric_comparison} and the previous bounds, \(g\) is bounded above and below (as a bilinear form) by \((1-Kr^2)\)
    (respectively \(1+Kr^2\)) times the euclidean inner product.
    This implies that \(\det(g)\) lies between \((1-Kr^2)^d\) and \((1+Kr^2)^d\), from which the claim follows.
    (Note that for symmetric matrices \(A,B\) with \(0 < A \leq B\), it holds that \(\det(A) \leq \det(B)\).
    Indeed, when \(A\) and \(B\) are diagonal, this is clear, and any two symmetric, positive definite matrices are simultaneously diagonizable as quadratic forms.)
\end{proof}

\section{Poisson--Voronoi cells are bounded}\label{app:bounded cells}

In this section, we show the following:

\begin{theorem}\label{thm:bdd-cells_continuous}
    Assume that there are positive constants \(r_0\) and \(b_0\) such that \(\mu(B_M(x,r_0)) > b_0\) for any \(x \in M\).
    Let \(\eta \sim \PPP(M,\mu)\).
    Then, almost surely, all cells in the Voronoi diagram induced by \(\eta\) are compact.
\end{theorem}

\begin{remark}\label{rem:bdd-cells_curvature}
    If all sectional curvatures of \(M\) are uniformly bounded from above and \(M\) has a positive global injectivity radius, then it follows from \eqref{eq:lower_bound_volume} that the conditions of the previous theorem are satisfied for \(\mu = \lambda \mu_M\), \(\lambda > 0\).
    Curiously, we do not need a lower bound on the curvature.
\end{remark}

\begin{remark}
    In \(\R^d\) it holds that the Voronoi diagram induced by a locally finite point set \(X\) contains no unbounded cells if and only if every point of \(X\) lies in the interior of the convex hull \(\operatorname{conv}(X)\) (cf.~\cite[Theorem 3.2.9]{voigt_diss}).
    In particular, this is the case when \(\operatorname{conv}(X) = \R^d\).
    This no longer holds true for general manifolds, in \(\H^d\) for example the appropriate condition on \(X\) is that there is no open horoball \(B\) such that \(B \cap X = \emptyset\) and \(\overline B \cap X \neq \emptyset\).
\end{remark}

\begin{remark}
    When $M$ has infinite diameter but finite volume, then trivially the Voronoi diagram associated to \(\eta \sim \PPP(M,\lambda \mu_M)\), \(\lambda > 0\), will almost surely contain at least one unbounded cell.
    In a similar vein, the same phenomenon occurs for manifolds with finite-volume ``spikes'' such as \(M \coloneqq \{(x,y,z) \in \R^3 \colon x^2+y^2=e^{-z}\}\).
\end{remark}

Before we give the proof, we consider a related discrete version of the problem.
The proof of the discrete version will make the continuous counterpart easier to understand.

Let \(G = (V,E)\) be a connected graph with vertex set \(V\) and edge set \(E\).
Assume that \(V\) is countable and that each vertex has finite degree.
Fix a parameter \(p\in(0,1)\) and independently declare each vertex \emph{occupied} with probability \(p\) and \emph{unoccupied} otherwise.
Let \(X\) be the set of occupied vertices.
In the \emph{graph Voronoi diagram}, the cell of an occupied vertex \(x\) is given by
\[ C(x;X) \coloneqq \{y \in V \colon \d(y,x) \leq \d(y,X)\}, \]
where \(\d(x,y)\) denotes the graph distance, i.e., the length of a shortest path from \(x\) to \(y\).
In the following, we also use the notation \(B(x,n)\) for a ball with center \(x\) and radius \(r\) w.r.t.\ the graph distance.

\begin{theorem}
    \label{thm:bdd-cells_discrete}
    Almost surely, all cells of the graph Voronoi diagram are finite.
\end{theorem}
\begin{proof}
    Fix a vertex \(x_0 \in V\).
    Let
    \[ C \coloneqq \{y \in G \colon \d(y,x_0) \leq \d(y,X)\}.\]
    If \(x_0\) is occupied, then \(C\) is the Voronoi cell of \(x_0\).
    We show that \(C\) is almost surely bounded, the claim then follows by a union bound.

    For \(n \in \N_0\) let \(r(n) \coloneqq (n+2)k,\) where \(k = k(p) \in \N\) is a constant to be specified later.
    We define annuli
    \[A_0 \coloneqq \{y \in V \colon \d(x_0,y) \leq r(0)\} \quad \text{and} \quad A_{n+1} \coloneqq \{y \in V \colon r(n) < \d(x_0,y) \leq r(n+1)\},\]
    as well as `spheres'
    \[S_n \coloneqq \{y \in V \colon \d(x_0,y) = r(n)\}\]
    for \(n \in \N_0\).
    Further, for any \(n \in \N_0\) and \(x\in S_n\) we define
    \[A_{n+1}(x) \coloneqq \{y \in A_{n+1} \colon \d(x,y) \leq k\} \quad \text{and} \quad
    S_{n+1}(x) \coloneqq \{y \in S_{n+1} \colon \d(x,y) \leq k\}.\]
    Note that \(A_{n+1} = \cup_{x \in S_n} A_{n+1}(x)\) and \(S_{n+1} = \cup_{x \in S_n} S_{n+1}(x)\).

    We now define an exploration that reveals the state of the vertices in \(A_{n}\) in step \(n\) and gives us sets \(C_n \subset S_n\)
    that contain \(C \cap S_n\).
    We use the fact that the cardinality of \(C_n\) is dominated by the number of children at level \(n\) of a subcritical inhomogeneous branching process to obtain that \(C\) is almost surely bounded.

    We inductively define the sets \(C_n \subset S_n\), \(n\in\N_0\) by setting \(C_0 \coloneqq S_0\) and
    \[ C_{n+1} \coloneqq \bigcup_{x \in C_n,\, A_{n+1}(x) \text{ unoccupied}} S_{n+1}(x), \]
    where we call \(A_{n+1}(x)\) unoccupied when all of its vertices are unoccupied.

    We prove inductively that \(C \cap S_n \subset C_n\).
    The case \(n=0\) is clear.
    For the induction step, let \(y \in C \cap S_{n+1}\).
    Take an \(x\in S_n\) with \(\d(x,y)=k\), so in particular it holds that \(y \in S_{n+1}(x)\).
    Since \(y \in C\), it follows that \(B(y,r(n+1)-1)\) is unoccupied, which implies that \(B(x,r(n)-1) \subset B(y,r(n+1)-1)\) is unoccupied.
    At the one hand, this implies that \(x \in C \cap S_n\), and hence, by the induction hypothesis, \(x \in C_n\).
    On the other hand, since \(k \leq r(n)-1\), it follows that \(A_{n+1}(x)\) is unoccupied.
    Combining those observations, we get that \(y\in C_{n+1}\) and since \(y\) was arbitrary, \(C \cap S_{n+1} \subset C_{n+1}\) follows.

    Notice that if \(y \in C\), then any shortest path from \(x_0\) to \(y\) is contained in \(C\).
    So if \(C\) is unbounded, then \(C \cap S_n \neq \emptyset\) for any \(n \in \N_0\).
    We can thus bound
    \[ \P(\text{\(C\) unbounded}) \leq \P(C \cap S_n \neq \emptyset) \leq \P(C_n \neq \emptyset) \leq \E|C_n| \]
    for any \(n \in \N_0\).
    Note that for \(n \in \N_0\)
    \begin{align*}
        \E|C_{n+1}|
            &\leq \E \sum_{x \in C_n,\, A_{n+1}(x) \text{ unoccupied}} |S_{n+1}(x)|\\
            &= \sum_{x\in S_n} \P(x \in C_n,\, A_{n+1}(x) \text{ unoccupied})\, |S_{n+1}(x)|\\
            &= \sum_{x\in S_n} \P(x \in C_n) \P(A_{n+1}(x) \text{ unoccupied})\, |S_{n+1}(x)|\\
            &= \sum_{x\in S_n} \P(x \in C_n) (1-p)^{|A_{n+1}(x)|} |S_{n+1}(x)|\\
            &\leq \sum_{x\in S_n} \P(x \in C_n) (1-p)^{(k-1)+|S_{n+1}(x)|} |S_{n+1}(x)|\\
            &\leq \left(\sup_{l\in\N_0} l(1-p)^l \right) \cdot (1-p)^{k-1} \cdot \sum_{x\in S_n} \P(x \in C_n)\\
            &= \left(\sup_{l\in\N_0} l(1-p)^l \right) \cdot (1-p)^{k-1} \cdot \E|C_n|.
    \end{align*}
    In the third line we used that the event \(\{x\in C_n\}\) only depends on vertices in \(B(x_0,r(n))\), which is disjoint from \(A_{n+1}(x)\).
    In the fifth line, we used that when \(S_{n+1}(x) \neq \emptyset\), then \(|A_{n+1}(x)| \geq (k-1) + |S_{n+1}(x)|\).
    Indeed, if \(y \in S_{n+1}(x)\), then there exists a shortest path \(x = z_0,z_1,\ldots,z_k = y\) from \(x\) to \(y\).
    It is not hard to see that the vertices \(z_1,\ldots,z_{k-1}\) are contained in \(A_{n+1}(x) \setminus S_{n+1}(x)\).
    Obviously, \(S_{n+1}(x)\) is also contained in \(A_{n+1}(x)\), hence this set has at least \((k-1) + |S_{n+1}(x)|\) elements.

    Since \(\sup_{l\in\N_0} l(1-p)^l < \infty\) we can choose \(k\) sufficiently large (depending on \(p\)) such that \(\E|C_{n+1}| \leq \frac{1}{2} \E|C_n|\) for all \(n\in\N_0\).
    Hence,
    \[ \P(\text{\(C\) unbounded}) \leq \lim_{n\to\infty} \E|C_n| = 0. \]
\end{proof}

Theorem \ref{thm:bdd-cells_continuous} can be proved by adapting the previous argument to the continuous setting.
Before stating the proof, we recall the star-convexity of the Voronoi cells.

\begin{lemma}\label{lem:conv}
    For any locally finite point set $\eta \subset M$, $x\in \eta$ and $u\in C(x;\eta)$, any minimizing geodesic between $u$ and $x$ is contained in $C(x;\eta)$.
\end{lemma}

\begin{proof} Let $u\in C(x;\eta)$ and let $w$ be a point on a minimizing geodesic between $u$ and $x$. Let $z\in\eta$. We have by the triangle inequality
	\[\d_M(u,w)+\d_M(w,z)\ge \d_M(u,z)\ge \d_M(u,x)=\d_M(u,w)+\d_M(w,x)\]
	yielding that $\d_M(w,x)\le \d_M(w,z)$ and $w\in C(x;\eta)$.
\end{proof}}

\begin{proof}[Proof of Theorem \ref{thm:bdd-cells_continuous}]
    Since $M$ is complete, any closed and bounded subset $K\subset M$ is compact (this is a consequence of the Hopf--Rinow Theorem, cf.~\cite[Corollary I.7.1.]{Chavel}).
    The closedness of the Voronoi cells is obvious from their definition, so it remains to show that they are almost surely bounded.
    By the Mecke formula (see, e.g., \cite[Theorem 4.1]{LastPenrosePP}), it holds that
    \[ \E \sum_{x \in \eta} \one\{C(x;\eta) \text{ unbounded}\} = \int_M \P(C(x;\eta) \text{ unbounded}) \mu(dx). \]
    It thus suffices to show that for any fixed \(x_0 \in M\),
    \[ \P(C(x_0;\eta) \text{ unbounded}) = 0 \]
    holds.
    In the following, we thus consider a fixed \(x_0\).
    For any \(x \in M\) it holds that
    \[ \P(B_M(x,r_0) \cap \eta = \emptyset) = e^{-\mu(B_M(x,r_0))} \leq e^{-b_0} \eqqcolon q. \]
    For \(n \in \N_0\) let
    \( r(n) \coloneqq (n+10)  k  r_0, \)
    where \(k = k(q) \geq 10\) is a constant to be specified later.
    Further let
    \[ A_0 \coloneqq B_M(x_0, r(0) + r_0) \quad \text{and} \quad A_{n+1} \coloneqq B_M(x_0, r(n+1)+r_0) \setminus B_M(x_0, r(n)+r_0), \]
    as well as
    \[ S_n \coloneqq \{y \in M \colon \d_M(x_0,y) = r(n)\}. \]
    To adapt the previous argument, we need to discretize the problem.
    To this end, for each \(n \in \N_0\), we pick a maximal subset \(S'_n \subset S_n\) such that \(\d_M(x',y')\geq 2r_0\)
    for any distinct \(x',y' \in S'_n\).
    Since \(S_n\) is compact, \(S'_n\) is  finite.
    For any \(n\in\N_0\) and \(x' \in S'_n\), we then define
    \[ A_{n+1}(x') \coloneqq A_{n+1} \cap B_M(x',(k+5)r_0) \quad \text{and} \quad
    S'_{n+1}(x') \coloneqq \{y' \in S'_{n+1} \colon B_M(y',r_0) \subset A_{n+1}(x')\}. \]
    We then inductively define sets \(C_n \subseteq S'_n\), \(n\in\N_0\), by setting \(C_0 \coloneqq S'_0\) and 
    \[ C_{n+1} \coloneqq \bigcup_{x' \in C_n,\, A_{n+1}(x') \cap \eta = \emptyset} S'_{n+1}(x'). \]
    We now show that for all \(n \in \N_0\) it holds that
    \begin{equation}\label{eq:inclusion_inductive}
        \{ y' \in S'_n \colon \exists y \in C(x_0;\eta) \cap S_n: \d_M(y',y)<2r_0 \} \subset C_n.
    \end{equation}
    The proof is by induction.
    The case \(n=0\) is trivial.
    For the induction step, take \(y' \in S'_{n+1}\) so that \(\d_M(y',y) <2r_0\) for some \(y \in C(x_0;\eta) \cap S_{n+1}\).
    Let \(\gamma\) be a minimizing geodesic joining \(x_0\) and \(y\).
    By Lemma \ref{lem:conv}, \(\gamma\) is contained in \(C(x_0;\eta)\).
    Taking the intersection of \(\gamma\) with \(S_n\) gives us a point \(x \in C(x_0;\eta) \cap S_n\) with \(\d_M(x,y) = k r_0\).
    By construction, there exists a point \(x' \in S'_n\) with \(\d_M(x',x)<2r_0\).
    This point belongs to \(C_n\), by the induction hypothesis.
    By the triangle inequality, \(\d_M(x',y') < (k+4)r_0\), from which we obtain \(y' \in S'_{n+1}(x')\).
    Further, \(y \in C(x_0;\eta) \cap S_{n+1}\) implies that \(\eta\cap B_M(y,r(n+1)) = \emptyset\).
    Using the triangle inequality again, we see that \(\d_M(x',y) < (k+2)r_0\), which yields
    \[ A_{n+1}(x') \subset B_M(x',(k+5)r_0) \subset B_M(x', r(n+1) - (k+2)r_0) \subset B_M(y,r(n+1)). \]
    Thus, \(A_{n+1}(x') \cap \eta = \emptyset\).
    We then conclude that \(y' \in C_{n+1}\).
    Since \(y'\) was arbitrary, \eqref{eq:inclusion_inductive} follows.

    Let \(n \in \N_0\), \(x'\in S_n\) and assume that \(S'_{n+1}(x') \neq \emptyset\).
    We now explain how to get an upper bound on \(\P(A_{n+1}(x') \cap \eta = \emptyset)\).
    Take \(y' \in S'_{n+1}(x')\) and let \(\gamma\) be a minimizing geodesic joining \(x'\) and \(y'\).
    Since \(\d_M(x_0,x') = r(n)\) and \(\d_M(x_0,y')=r(n+1)\), we have that \(\d_M(x',y')\geq k r_0\).
    Further, by the definition of \(S'_{n+1}(x')\), it holds that \(\d_M(x',y') \leq (k+5)r_0\).
    For \(m=1,\ldots,\lfloor(k-7)/2\rfloor\), let \(z_m\) be the point on \(\gamma\) with
    \[ \d_M(x',z_m) = (5+2m)r_0. \]
    It then holds that
    \[ \d_M(x_0, z_m) \leq \d_M(x_0,x') + \d_M(x',z_m) \leq r(n) + (k-2)r_0 = r(n+1) -2r_0, \]
    and
    \[\begin{aligned}
        \d_M(x_0, z_m) \geq \d_M(x_0,y') - \d_M(y',z_m) &= \d_M(x_0,y') - (\d_M(x',y')-\d_M(x',z_m))\\
        &\geq r(n+1) - (k+5)r_0 + 7 r_0 = r(n) + 2 r_0.
    \end{aligned}\]
    It is then easy to verify that the balls \(B_M(z_1,r_0),\ldots,B_M(z_m,r_0)\) are disjoint and contained in \(A_{n+1}(x')\).
    They are further disjoint from the balls \(B_M(y'', r_0)\), \(y'' \in S'_{n+1}(x')\), which are also contained in  \(A_{n+1}(x')\).
    From the basic properties of the Poisson process it follows that
    \[ \P(A_{n+1}(x') \cap \eta = \emptyset) \leq q^{\lfloor(k-7)/2\rfloor + |S'_{n+1}(x')|} \leq q^{(k-9)/2 + |S'_{n+1}(x')|}. \]
    For \(n \in \N_0\) we now get
    \begin{align*}
        \E|C_{n+1}|
        &\leq \E \sum_{x' \in C_n,\, A_{n+1}(x') \cap \eta = \emptyset} |S'_{n+1}(x')|\\
        &= \sum_{x'\in S_n'} \P(x' \in C_n,\, A_{n+1}(x') \cap \eta = \emptyset)\, |S'_{n+1}(x')|\\
        &= \sum_{x'\in S_n'} \P(x' \in C_n) \P(A_{n+1}(x') \cap \eta = \emptyset)\, |S'_{n+1}(x')|\\
        &\leq \sum_{x'\in S_n'} \P(x' \in C_n) q^{(k-9)/2 + |S'_{n+1}(x')|} |S'_{n+1}(x')|\\
        &\leq \left(\sup_{l\in\N_0} l q^l \right) \cdot q^{(k-9)/2} \cdot \sum_{x'\in S'_n} \P(x' \in C_n)\\
        &= \left(\sup_{l\in\N_0} l q^l \right) \cdot q^{(k-9)/2} \cdot \E|C_n|,
    \end{align*}
    where we used that the events \(\{x' \in C_n\}\) and \(\{A_{n+1}(x') \cap \eta = \emptyset\}\) are independent by construction.
    As in the proof of Theorem \ref{thm:bdd-cells_discrete}, we then get \(\lim_{n\to\infty} \E|C_n| = 0\) when \(k\) is sufficiently large.
    Note that \(C(x_0;\eta)\) is star-shaped w.r.t.\ \(x_0\), so in particular, when it is not contained in \(B_M(x_0,r(n))\),
    then it intersects \(S_n\) in some point \(y\).
    By construction, this implies that there is a point \(y' \in S'_n\) with \(\d_M(y',y)< 2r_0\), and hence \(y' \in C_n \neq \emptyset\).
    Hence, we get that
    \[ \P(C(x_0;\eta) \subset B_M(x_0,r(n))) \geq 1 - \P(C_n \neq \emptyset) \geq 1 - \E|C_n| \to 1, \]
    as \(n \to \infty\).
\end{proof}

\section{Non-trivial phase transition}\label{app:phase-transition}

In this section, we show that if $M$ satisfies the assumptions of Theorem~\ref{thm:main} and has superlinear growth, then $$0<p_c(M,\lambda)\leq p_u(M,\lambda)<1,$$ for \emph{all} $\lambda>0$. The proof is based on renormalization arguments, which bear similarities with those used in Section~\ref{sec:fine_graining}. However, since here we need to deal with arbitrarily low intensities, the embedded graphs used are taken to be sparse, i.e.~$G_r$ for $r$ large. For that reason, we say that it is a coarse-graining argument, in contrast with the fine-graining one of Section~\ref{sec:fine_graining}. An important difference here is the fact that the degrees of $G_r$ can diverge (even exponentially) as $r\to \infty$, forcing us to have a quantitative control of the probabilities of the local events used in the renormalization. This is also why we do not rely on \cite{LigSchSta97} here.

For $x\in M$ and $r>0$, define $$v(x,r) \coloneqq \mu_M\big(B_M(x,r)\big).$$ Also, let $$\overline{v}(r) \coloneqq \sup_{x\in M}v(x,r) \quad \text{and} \quad \underline{v}(r) \coloneqq \inf_{x\in M} v(x,r).$$
In this section, we assume  that $M$ has superlinear growth, meaning that 
\begin{equation}
\label{eq:superlinear}
    \lim_{r\rightarrow \infty} \underline{v}(r)/r =\infty. 
\end{equation}
We will use the fact that if $M$ satisfies conditions \ref{thm cond: global injectivity radius} and \ref{thm cond: uniform bound on curvature} of Theorem \ref{thm:main}, then by \eqref{eq:upper_bound_volume} there exists $c>0$ such that 
\begin{equation}\label{eq:atmostexpgrowth}
   \overline{v}(r)\leq \exp(cr), \quad \text{for all \(r \geq 0\)}.
\end{equation}

Recall that $\overline{\eta}(\lambda,p)=(\eta_w(\lambda,p),\eta_b(\lambda,p))$ denotes a pair of independent Poisson point processes with intensities $p\lambda \mu_M$ and $(1-p)\lambda \mu_M$, respectively, and $\pazocal{V}_M(\lambda,p)=\pazocal{V}_M(\overline{\eta}(\lambda,p))$ is the white region of the corresponding (colored) Voronoi tessellation. Also recall the embedded graphs  $G_r$ defined in Section \ref{sec:geometry_G_ep}, which satisfy the property that any point $x\in M$ is within distance $2r$ of a vertex of $G_r$ and every pair of vertices in $G_r$ are at least $r$ apart.

We start by proving the existence of a subcritical phase -- see \cite[Lemma 6.5]{BS01} for a similar argument, showing that $p_c(\H^2,\lambda)>0$ for all $\lambda>0$.

\begin{theorem}\label{thm:p_c>0}
    Suppose $M$ satisfies conditions \ref{thm cond: global injectivity radius} and \ref{thm cond: uniform bound on curvature} of Theorem \ref{thm:main} as well as the superlinear growth condition \eqref{eq:superlinear}. Then for all $\lambda>0$ we have $p_c(M, \lambda)>0$.
\end{theorem}

\begin{proof}
Fix $\lambda>0$ and a large enough constant $r>0$ to be determined later. For $x\in M$, consider the event 
$$A_x\coloneqq \{ \eta_b(\lambda,p)\cap B_M(x,r) = \emptyset \} \cup \{ \eta_w(\lambda,p) \cap B_M(x,5r) \neq \emptyset \}.$$
By the definition of \(\overline{\eta}(\lambda,p)\),
$$\mathbb{P}(\eta_b(\lambda,p)\cap B_M(x,r) = \emptyset)=e^{-\lambda(1-p) v(x,r)} \leq e^{-\lambda(1-p)\underline{v}(r)},$$
and
$$\mathbb{P}(\eta_w(\lambda,p) \cap B_M(x,5r) \neq \emptyset)=1-e^{-\lambda p v(x,5r)} \leq 1-e^{-\lambda p \overline{v}(5r)}.$$
Therefore, we can choose $p=p(r)>0$ small enough such that, for every $x\in M$,
\begin{equation}\label{eq:bound_Ax_sub}
	\mathbb{P}(A_x)\leq e^{-\frac{\lambda}{2} \underline{v}(r)}.
\end{equation}

Note that on the complementary event $A_x^c$, every cell intersecting $B_M(x,2r)$ is black. In particular, if $\pazocal{V}_M(\lambda,p)$ has an unbounded white connected component, then the site percolation model on $G_r$ given by
$$Z(x) \coloneqq \mathbbm{1}_{A_x},~~~x\in G_r,$$
percolates. Therefore, it is sufficient to show that $Z$ does not percolate to conclude that $p_c(\lambda)\geq p>0$.
    
Now, fix $x_0\in G_r$ and assume that $x_0$ belongs to an infinite component of $Z$. Then there exists an infinite self-avoiding path $x_0,x_1,\ldots$ in $G_r$ which is $Z$-open. In order to use independence, we now extract an ``almost-path'' with mutual distance at least $10r$. Define $(k_i)_{i\geq 0}$ inductively by $k_0=0$ and for $i\geq 0$, $k_{i+1}=\sup \{k\geq 0: \d_M(x_{k},x_{k_i})\leq 14r\}$. Set $y_i=x_{k_i}$. Then we certainly have for all $i\geq 0$, $\d_M(y_i, y_{i+1})\leq 14r$ and $\d_M(y_i, y_j)\geq 10r$ whenever $i\neq j$. In particular, all events $A_{y_i}$, \(i \geq 0\), are independent.  For $n\geq 1$, the number of possible sequences $y_1,\ldots, y_n$ as above is at most 
$$\big(\sup_{y\in M} |B_M(y,14r)\cap G_r|\big)^n\leq \left(\frac{\overline{v}(15r)}{\underline{v}(r)}\right)^n.$$ 
Therefore, by \eqref{eq:bound_Ax_sub} and a union bound, we obtain 
    \begin{equation}
        \mathbb{P}(\lr{}{Z}{x_0}{\infty})
        \leq \left(\frac{\overline{v}(15r)}{\underline{v}(r)}\right)^n  e^{-\frac{\lambda}{2} \underline{v}(r)n},
    \end{equation}
    for all $n\geq 1$. By the superlinear and subexponential growth bounds \eqref{eq:superlinear} and \eqref{eq:atmostexpgrowth}, we can take $r$ large enough so that $\frac{\overline{v}(15r)}{\underline{v}(r)} e^{-\frac{\lambda}{2} \underline{v}(r)}<1$, thus concluding the proof.
\end{proof}

\begin{theorem}\label{thm:p_c<1}
    Suppose $M$ satisfies all four conditions of Theorem \ref{thm:main} as well as the superlinear growth condition \eqref{eq:superlinear}. Then for all $\lambda>0$ we have $p_u(M, \lambda)<1$.
\end{theorem}

\begin{proof}
    By Theorem~\ref{thm:bound on p_u for Voronoi} and Theorem~\ref{thm:p_c>0}, it is enough to show that $p_c(M,\lambda)<1$. Fix $\lambda>0$ and a large enough constant $r>0$ to be determined later. For $x\in M$, consider the event 
    $$A_x\coloneqq \{ \eta_w(\lambda,p)\cap B_M(x,r) \neq \emptyset \} \cap \{ \eta_b(\lambda,p) \cap B_M(x,5r) = \emptyset \}.$$
    As in the proof of Theorem~\ref{thm:p_c>0},  we can choose $p=p(r)<1$ close enough to $1$ so that 
    \begin{equation}\label{eq:bound_Ax_sup}
    	\mathbb{P}(A_x)\geq 1-e^{-\frac{\lambda}{2} \underline{v}(r)}.
    \end{equation}

	Note that on the event $A_x$, every cell intersecting $B_M(x,2r)$ is white. In particular, if the site percolation model on $G_r$ given by
	$$Z(x) \coloneqq \mathbbm{1}_{A_x},~~~x\in G_r,$$
	percolates, then $\pazocal{V}_M(\lambda,p)$ has an unbounded white connected component. Therefore, it is sufficient to show that $Z$ does percolate to conclude that $p_c(\lambda)\leq p<1$.
    
    Now, fix $x_0\in G_r$, $n\geq1$, and assume that every component of $Z$ intersecting $B_{G_r}(x_0,n)$ is finite. Then there exists a $Z$-closed minimal cutset separating $B_{G_r}(x_0,n)$ from infinity in $G_r$. By proceeding as in the proof Lemma~\ref{lem:cutsets->paths} and applying the reduction procedure in the proof of Theorem~\ref{thm:p_c>0}, there exist constants $\tilde{c}_0$, $\tilde{C}_0$ depending on $r$ and $x_0$, a constant $\tilde{C}_1$ independent of $r$ and $x_0$, and subsets $(\tilde{V}^{x_0}_k)_{k\geq1}$ of $V_r$ with $|\tilde{V}^{x_0}_k|\leq \tilde{C}_0e^{\tilde{C_1}rk}$, such that the following holds:  every minimal cutset separating $B_{G_r}(x_0,n)$ from infinity contains a sequence $y_0,y_1,\ldots,y_{k}$ satisfying $y_0\in \tilde{V}^{x_0}_k$ with $k\geq \tilde{c}_0 \log n-\tilde{C}_0$, $\d_M(y_i, y_{i+1})\leq 14r$ for all $i$ and $\d_M(y_i, y_j)\geq 10r$ whenever $i\neq j$. As in the proof of Theorem~\ref{thm:p_c>0}, for every $k\geq 1$ and $y_0\in \tilde{V}_k^{x_0}$, the number of possible sequences $y_1,\ldots, y_k$ as above is at most $\left(\frac{\overline{v}(15r)}{\underline{v}(r)}\right)^k$, and for each such sequence the events \(A_{y_i}\), \(1 \leq i \leq k\), are independent. 
    Therefore, by \eqref{eq:bound_Ax_sup} and a union bound, we obtain 
    \begin{equation}
    	\mathbb{P}(\nlr{}{Z}{B_{G_r}(x_0,n)}{\infty})
    	\leq \sum_{k\geq  \tilde{c}_0 \log n-\tilde{C}_0} \tilde{C}_0e^{\tilde{C_1}rk} \left(\frac{\overline{v}(15r)}{\underline{v}(r)}\right)^k  e^{-\frac{\lambda}{2} \underline{v}(r)k},
    \end{equation}
    for all $n\geq 1$. By the superlinear and subexponential growth bounds \eqref{eq:superlinear} and \eqref{eq:atmostexpgrowth}, we can take $r$ large enough so that $e^{\tilde{C_1}r}\frac{\overline{v}(15r)}{\underline{v}(r)} e^{-\frac{\lambda}{2} \underline{v}(r)}<1$, and then $n$ large enough so that the above sum is strictly smaller than $1$, thus concluding the proof.
\end{proof}

\bibliographystyle{alpha}

\end{document}